\documentclass[11pt,reqno,centertags]{amsart}

\usepackage{amsfonts}
\usepackage{color,enumitem,graphicx}

  \usepackage{amsmath,amsthm,amsfonts,amssymb}
  \usepackage{soul}
  \usepackage{ulem}
\begin{document}

\title{Fractional Gross-Pitaevskii equations
in non-Gaussian attractive Bose-Einstein condensates}
\date{}
\maketitle

\vspace{ -1\baselineskip}

{\small
\begin{center}

{\sc  Jinge Yang}  \\
 School of Science $\&$ Key Laboratory of Engineering Mathematics and Advanced Computing,
Nanchang Institute of Technology\\
Nanchang 330099,
P.~R.~China\\
email: Jinge Yang: jgyang2007@yeah.net\\[10pt]

{\sc  Jianfu Yang } \\
Department of Mathematics,
Jiangxi Normal University\\
Nanchang, Jiangxi 330022,
P.~R.~China\\
email: Jianfu Yang: jfyang200749@sina.com\\[10pt]
\end{center}
}

\renewcommand{\thefootnote}{}

\begin{quote}
{\bf Abstract.}  In this paper, we investigate normalized solutions of a fractional Gross-Pitaevskii equation, which arises in an  attractive Bose-Einstein condensation  consisting of $N$ bosons
moving by L\'{e}vy flights.  We prove that there exists a positive constant $N^*$,  such that if $0<N<N^*$ and  the L\'{e}vy index  $\alpha$ closed to $2$, the fractional Gross-Pitaevskii equation admits a local minimal normalized solution $u_\alpha$ and a mountain pass solution $v_\alpha$, but there does not exist positive local minimal solution if $N>N^*$ and $\alpha$ closed to $2$. We also study the asymptotic behavior of $u_\alpha$ and $v_\alpha$  as $\alpha\to 2_-$.

\end{quote}

\newcommand{\N}{\mathbb{N}}
\newcommand{\R}{\mathbb{R}}
\newcommand{\Z}{\mathbb{Z}}

\newcommand{\cA}{{\mathcal A}}
\newcommand{\cB}{{\mathcal B}}
\newcommand{\cC}{{\mathcal C}}
\newcommand{\cD}{{\mathcal D}}
\newcommand{\cE}{{\mathcal E}}
\newcommand{\cF}{{\mathcal F}}
\newcommand{\cG}{{\mathcal G}}
\newcommand{\cH}{{\mathcal H}}
\newcommand{\cI}{{\mathcal I}}
\newcommand{\cJ}{{\mathcal J}}
\newcommand{\cK}{{\mathcal K}}
\newcommand{\cL}{{\mathcal L}}
\newcommand{\cM}{{\mathcal M}}
\newcommand{\cN}{{\mathcal N}}
\newcommand{\cO}{{\mathcal O}}
\newcommand{\cP}{{\mathcal P}}
\newcommand{\cQ}{{\mathcal Q}}
\newcommand{\cR}{{\mathcal R}}
\newcommand{\cS}{{\mathcal S}}
\newcommand{\cT}{{\mathcal T}}
\newcommand{\cU}{{\mathcal U}}
\newcommand{\cV}{{\mathcal V}}
\newcommand{\cW}{{\mathcal W}}
\newcommand{\cX}{{\mathcal X}}
\newcommand{\cY}{{\mathcal Y}}
\newcommand{\cZ}{{\mathcal Z}}

\newcommand{\abs}[1]{\lvert#1\rvert}
\newcommand{\xabs}[1]{\left\lvert#1\right\rvert}
\newcommand{\norm}[1]{\lVert#1\rVert}

\newcommand{\loc}{\mathrm{loc}}
\newcommand{\p}{\partial}
\newcommand{\h}{\hskip 5mm}
\newcommand{\ti}{\widetilde}
\newcommand{\D}{\Delta}
\newcommand{\e}{\epsilon}
\newcommand{\bs}{\backslash}
\newcommand{\ep}{\emptyset}
\newcommand{\su}{\subset}
\newcommand{\ds}{\displaystyle}
\newcommand{\ld}{\lambda}
\newcommand{\vp}{\varphi}
\newcommand{\wpp}{W_0^{1,\ p}(\Omega)}
\newcommand{\ino}{\int_\Omega}
\newcommand{\bo}{\overline{\Omega}}
\newcommand{\ccc}{\cC_0^1(\bo)}
\newcommand{\iii}{\opint_{D_1}D_i}

\theoremstyle{plain}
\newtheorem{Thm}{Theorem}[section]
\newtheorem{Lem}[Thm]{Lemma}
\newtheorem{Def}[Thm]{Definition}
\newtheorem{Cor}[Thm]{Corollary}
\newtheorem{Prop}[Thm]{Proposition}
\newtheorem{Rem}[Thm]{Remark}
\newtheorem{Ex}[Thm]{Example}

\numberwithin{equation}{section}
\newcommand{\meas}{\rm meas}
\newcommand{\ess}{\rm ess} \newcommand{\esssup}{\rm ess\,sup}
\newcommand{\essinf}{\rm ess\,inf} \newcommand{\spann}{\rm span}
\newcommand{\clos}{\rm clos} \newcommand{\opint}{\rm int}
\newcommand{\conv}{\rm conv} \newcommand{\dist}{\rm dist}
\newcommand{\id}{\rm id} \newcommand{\gen}{\rm gen}
\newcommand{\opdiv}{\rm div}

\vskip 0.2cm \arraycolsep1.5pt
\newtheorem{Lemma}{Lemma}[section]
\newtheorem{Theorem}{Theorem}[section]
\newtheorem{Definition}{Definition}[section]
\newtheorem{Proposition}{Proposition}[section]
\newtheorem{Remark}{Remark}[section]
\newtheorem{Corollary}{Corollary}[section]

\section{Introduction}

Bose-Einstein condensation (BEC) of an ideal gas, predicted by Einstein and Bose  in
1924 \cite{Bose1924,Einstein1925}, is the paradigm of quantum
statistical phase transitions. This phase transition is manifested by an abrupt growth in the population of the
ground state of the potential confining a gas of bosons. Such phenomenon  was  realized experimentally  in 1995 \cite{Anderson1995,Davis1995},
it opened a new field in the study of macroscopic quantum
phenomena, including  superconductivity \cite{Barden1957} and macroscopic quantum effects \cite{Dalfovo1999}, and  etc. The dynamics of the condensate wave function $\psi$ for a gas is
well described by the  Gross-Pitaevskii equation \cite{Gross1961,Pitaevskii2003}. In scaled units, the  Gross-Pitaevskii equation is given by
\begin{equation}\label{eq:1.1}
i\psi_t(x,t)=-\Delta\psi(x,t)+V(x)\psi(x,t)-\kappa |\psi|^2\psi(x,t),\ \  (x,t)\in \mathbb{R}^n\times\mathbb{R},
\end{equation}
 where  $V$ is an external trapping potential.  The constant $\kappa\in\mathbb{R}\setminus\{0\}$ is the strength of interaction  between particles.
 The forces between the particles in the condensates is attractive if $\kappa>0$, and is repulsive if $\kappa<0$.

The dilute Bose gas in the low-temperature
limit becomes almost completely condensed,  the number of particles $N$ in the
 condensate is
approximately equal to the total number of particles. Hence, the condensate wave function $\psi$ can be normalized to the
number of condensed particles,
\begin{equation}\label{eq:1.2}
\int_{\mathbb{R}^{n}}|\psi|^2\,dx=N.
\end{equation}
When  the  interaction is attractive, the  condensate  collapses if the particle number exceeds a critical value; see \cite{Dalfovo1999, Huepe1999} and etc.
Mathematically, it amounts to prove the existence and nonexistence of  the  solution  in the form $\psi(x,t)=e^{-i\mu_s t}u(x)$ of the Gross-Pitaevskii equation \eqref{eq:1.1} constrained by \eqref{eq:1.2}.  The function $u$ then satisfies
\begin{equation}\label{eq:1.3}
-\Delta u+V(x)u=\kappa |u|^2u+\mu_{{}_N} u, \ \  x\in \mathbb{R}^n
\end{equation}
and
\[
\int_{\mathbb{R}^{{}^n}}|u|^2\,dx=N.
\]
Such a solution $u$ is referred to the  normalized solution.

For the attractive BEC, that is, the strength of interaction $\kappa>0$,
the Gross-Pitaevskii (GP) functional of  \eqref{eq:1.3} in two dimensions
\begin{equation*}
J_{2,\kappa}(u)=\frac{1}{2}\int_{\mathbb{R}^2}\big(|\nabla u(x)|^2+V(x)|u(x)|^2\big)\,dx-\frac{\kappa }{4}\int_{\mathbb{R}^2}|u(x)|^4\,dx
\end{equation*}
is well defined in the space
\[
\mathcal{H}:=\Big\{u\in H^1(\mathbb{R}^2):\int_{\mathbb{R}^2}V(x)|u(x)|^2\,dx<\infty\Big\}.
\]
The functional $J_{2,\kappa}$ is critical  on
\[
S_0(N)=\{u\in \mathcal{H}:\int_{\mathbb{R}^2}|u|^2\,dx=N\}.
\]
This can be seen by considering the GP equation
 \begin{equation}\label{eq:1.4}
-\Delta u+V(x)u=\kappa |u|^pu+\mu_s u, \ \  x\in \mathbb{R}^n
\end{equation}
for $p>0$. If $V(x)=0$, we may verify by taking $u_\lambda(x)=\lambda^{\frac n2}u(\lambda x)$ that the functional
\[
J_{p,\kappa}(u_\lambda)=\frac{1}{2}\lambda^2\int_{\mathbb{R}^n}|\nabla u(x)|^2\,dx-\frac{\kappa }{p+2}\lambda^{\frac {np}2}\int_{\mathbb{R}^n}|u(x)|^{p+2}\,dx
\]
is bounded from below on $S_0(N)$ if $0<p<\frac4n$, or $p=\frac4n$ and $N$ is small enough. It is  unbound from below on $S_0(N)$ if $p>\frac4n$. These three cases are usually known as  the mass subcritical, critical  and  supercritical, respectively.

Normalized solutions of \eqref{eq:1.3} for $n=2$ can be obtained by investigating the constrained minimization problem
\begin{equation}\label{eq:1.5}
e_N:=\inf_{u\in S_0(N)}J_{2,\kappa}(u).
\end{equation}
It was proved in \cite{Guo2014} that for the trapping potential $V$,  including the harmonic potential, and $N=1$, there exists a critical value $\kappa^*>0$ such that problem \eqref{eq:1.3} admits a minimizer if $0<\kappa<\kappa^*$,  and there exists no minimizer  if $\kappa\geq \kappa^*$. The threshold value $\kappa^*$ is determined by the unique, up to translation, positive solution $Q$ of the nonlinear scalar field equation
\begin{equation}\label{eq:1.6}
-\Delta u+u=|u|^2u\ \ {\rm in} \ \ \mathbb{R}^2,\ \ u\in H^1(\mathbb{R}^2).
\end{equation}
Let $u=\sqrt{N}v$. Then $v\in S_0(1)$ and
\begin{equation*}
J_{2,\kappa}(u)=NJ_{2,\kappa N}(v)=N\Big[\frac{1}{2}\int_{\mathbb{R}^2}\big(|\nabla v(x)|^2+V(x)|v(x)|^2\big)\,dx-\frac{\kappa N }{4}\int_{\mathbb{R}^2}|u(x)|^4\,dx\Big].
\end{equation*}
Hence, for  $N>0$ and the trapping potential $V(x)$, there exists a critical value
\begin{equation}\label{eq:1.8}
N^*:=\|Q\|_{L^2(\mathbb{R}^2)}^2/\kappa>0
\end{equation}
 such that problem \eqref{eq:1.3} admits a minimizer if $0<N<N^*$,  and there exists no minimizer  if $N\geq N^*$.

Such a result was extended to various problems in BEC, see for example \cite{Meng-Zeng2022, Wang-Yang2018, WZZ2022, ZZ2019}.

In general, if  $0<p<2$, $J_{p,\kappa}(u)$ admits  at least a minimizer for all $N>0$, see \cite{Lions1984,Lions1984-2}. Furthermore,   qualitative properties of the minimizers of $J_{p,\kappa}(u)$, such as the uniqueness for $N$ small enough, the concentration as
$N\to+\infty$, were investigated  in
\cite{ Maeda2010} and references therein. For any fixed $N>0$, the limiting behavior of the minimizers of $J_{p,\kappa}(u)$ as $p\to 2_-$ was discussed in \cite{Guo-Zeng-Zhou2014, Z2017}.

In the mass supercritical case $p>2$,  since the GP functional is not bounded below, normalized solutions can not be found by studying minimization type problems for GP functionals. One then turns to look for critical points of GP functionals. Observing in this case that GP functionals have the mountain pass structure on the constraint, Jeanjean \cite{Jeanjean1997} developed a method to find  normalized solutions for $V=0$. Inspired by \cite{Jeanjean1997}, we  considered in \cite{Yang2020} mass supercritical problems with nontrivial potentials. For $V(x)=|x|^2$, the harmonic potential, and $n=2$, we proved in \cite{Yang2020} that equation  \eqref{eq:1.4} admits a local minimal  and a mountain pass solutions if $0<N<N^*$ and $p>2$ close to $2$.  In  \cite{Bellazzini2017}, Bellazzini et al investigated local minimal solutions of \eqref{eq:1.4} with partial confinement. Recently,  the mountain pass type normalized solutions of mass supercritical Schr\"{o}dinger equations  were studied under more general potentials, see for example \cite {Bartsch2021,Ding2022} and references therein.

In this paper, we consider the following fractional Gross-Pitaevskii equation
\begin{equation}\label{eq:1.8}
i\psi_t(x,t)=(-\Delta)^{\alpha/2}\psi(x,t)+V(x)\psi(x,t)-\kappa |\psi|^p\psi(x,t),\ \  (x,t)\in \mathbb{R}^n\times\mathbb{R}.
\end{equation}
It is known that the condensation of weakly interacting bosons moving by  Gaussian paths was  usually  described by the classical Gross-Pitaevskii equation \eqref{eq:1.1}. In recent years,  Bose-Einstein condensates not  obeying the Gaussian distribution law has been found \cite{Ertik2012}. The  motion of bosons in the  non-Gaussian Bose-Einstein condensates is realized by L\'{e}vy flights, instead of Gaussian path. In order to describe this new phenomenon,  equation \eqref{eq:1.8} was introduced in \cite{Sakaguchi2022, Uzar2013}. In \eqref{eq:1.8},{ $\alpha\in(0,2]$ is the  L\'{e}vy index $\alpha$}, which measures self-similarity in the L\'{e}vy fractal path. The operator $(-\Delta)^{\alpha/2}$ is defined by the Fourier transform as
\[
(-\Delta)^{\alpha/2}u = \mathcal{F}^{-1}(|\xi|^{\alpha}\mathcal{F}u),
\]
where $\mathcal{F}u$ denotes the Fourier transform of $u$.  In the sequel, we denote $s=\alpha/2$ for simplicity.

If $\psi(x,t)=e^{-i\mu_s t}u(x)$ is a standing wave solution of equation \eqref{eq:1.8}, then $u$ is  a solution of  the
 fractional Gross-Pitaevskii equation
\begin{equation}\label{eq:1.10}
(-\Delta)^su+V(x)u=\kappa |u|^pu+\mu_s u,\ \ {\rm in} \ \ \mathbb{R}^n,
\end{equation}
where $0<s<1$. Equation \eqref{eq:1.10} can be classified as before by the growth of nonlinearity, that is, if $0<p<\frac {4s}{n}$, $p=\frac {4s}{n}$ or $p>\frac {4s}{n}$,  equation \eqref{eq:1.10} is referred to be mass subcritical, critical and supercritical, respectively.
Recently, the research on problem \eqref{eq:1.10}  has been drawn much attention by taking $\kappa>0$  as the parameter. For the case $V=0$, the existence of solutions to equation \eqref{eq:1.10} and related problems were considered in \cite{Appolloni2021, Zou2022, Luo2020, Zhen2022} etc for different growths of nonlinearities; while for $V \not\equiv0$, one has to take into account the effect of the potential $V$ besides the nonlinearity.
Take into account these two factors, rich results in the research of problem \eqref{eq:1.10} emerged, see  for instance \cite{Du2019, He-Long2016, Peng2021, Zuo2023} and references therein.

In the study of Bose-Einstein condensates,   the  attractive fractional Gross-Pitaevskii equation with  the harmonic potential in two dimensions
 \begin{equation}\label{eq:1.11}
i\psi_t(x,t)=(-\Delta)^s\psi(x,t)+|x|^2\psi(x,t)-|\psi|^2\psi(x,t),\ \  (x,t)\in \mathbb{R}^2\times\mathbb{R}
\end{equation}
is of particular interest.

In this paper, we focus on problem \eqref{eq:1.11}. Correspondingly, $u$ satisfies
\begin{equation}\label{eq:1.12}
(-\Delta)^su+|x|^2u= |u|^2u+\mu_s u.
\end{equation}
For problems with general potentials $V(x)$, we discuss in the last section.

We remark that, in above mentioned works, the parameter $\kappa$ in \eqref{eq:1.10} plays an important role. The existence and collapse of solutions of \eqref{eq:1.10} may be determined by the range of $\kappa$. Once $\kappa$ is fixed, say $\kappa=1$ as in \eqref{eq:1.12}, many researches on \eqref{eq:1.12} may run into dilemma, usual powerful methods may not be applicable.

We use in this paper $s$ as the parameter. Physically, in a solid state physics realization of fractional quantum
mechanics, Stickler \cite{S2013} studied the limit of small wave numbers as the L\'{e}vy parameter
of order $\alpha=2s\to 2$. Mathematically, it proved by
Proposition 4.4 in \cite{Nezza2012} that  for any $u\in C^\infty_c(\mathbb{R}^n)$, the following statements hold:

$(i)$ $ \lim_{s\to0+}(-\Delta)^su = u;$

$(ii)$  $\lim_{s\to1_-}(-\Delta)^su =-\Delta u$.

On the other hand, we know from \cite{Frank-Lenzmann2016} that, if $\frac12<s<1$, there exists  a unique positive radial ground state solution $Q_s$ of the fractional scalar field equation
\begin{equation}\label{eq:1.13}
(-\Delta)^su+u=|u|^2u\ \ {\rm in} \ \ \mathbb{R}^2,\ \ u\in H^s(\mathbb{R}^2)
\end{equation}
up to translation. Moreover, if $s\to 1_-$, $Q_s$  converges in $L^2(\mathbb{R}^2)\cap L^4(\mathbb{R}^2)$ to the positive radial ground state solution $Q$
of  the equation
\begin{equation}\label{eq:1.14}
-\Delta u+u=|u|^2u\ \ {\rm in} \ \ \mathbb{R}^2,\ \ u\in H^1(\mathbb{R}^2).
\end{equation}
 Particularly, we have
\begin{equation}\label{eq:1.15}
\|Q_s\|_{L^2(\mathbb{R}^2)}\to \|Q\|_{L^2(\mathbb{R}^2)}
\end{equation}
as $s\to 1_-$.

These works enlighten us to explore a way to investigate problem \eqref{eq:1.12}. It seems that no references in this direction appeared in literature for nonlinear fractional Gross-Pitaevskii equations with potentials.

Now, let us define the associated energy functional
\begin{equation}\label{eq:1.16}
E_{N,s}(u)=\frac{1}{2}\int_{\mathbb{R}^2}\big(|(-\Delta)^{\frac s2} u(x)|^2+|x|^2|u(x)|^2\big)\,dx-\frac{1 }{4}\int_{\mathbb{R}^2}|u(x)|^4\,dx.
\end{equation}
of equation \eqref{eq:1.12}, which is well defined in the space
\[
\mathcal{H}_s:=\Big\{u\in H^s(\mathbb{R}^2):\int_{\mathbb{R}^2}|x|^2|u(x)|^2\,dx<\infty\Big\}.
\]

The cubic nonlinearity appeared in \eqref{eq:1.12} stems from usual physical consideration, such as the interaction of particles, or the Kerr effect. This exerts an influence on the GP functional $E_{N,s}(u)$  constrained on
\[
S(N)=\{u\in\mathcal{H}_s:\int_{\mathbb{R}^2}|u|^2\,dx=N\}.
\]
Indeed, $E_{N,s}(u)$ is unbounded from below on $S(N)$ for $0<s<1$.  This can be seen by letting $u_\lambda=\lambda\varphi(\lambda x)$ with $\varphi\in S(N)$, then $u_\lambda\in S(N)$  and
\[
E_{N,s}(u_\lambda)=\frac{1}{2}\lambda^{2s}\int_{\mathbb{R}^2}\big(|(-\Delta)^{\frac s2} u|^2+\lambda^{-2s-2}|x|^2|u|^2\big)\,dx-\frac{1}{4}\lambda^{2}\int_{\mathbb{R}^2}|u|^4\,dx
\]
approaches to $-\infty$, as $\lambda$ tends to $+\infty$.

Since $E_{N,s}(u)$ has no minimizer on $S(N)$, we turn to look for critical points of $E_{N,s}(u)$.
A critical point $u\in \mathcal{H}_s$ of $E_{N,s}(u)$ in fact is a weak solutions of \eqref{eq:1.12}, namely, $u$  satisfies
\begin{equation}\label{eq:1.17}
\int_{\mathbb{R}^2}(-\Delta)^{\frac{s}2}u(-\Delta)^{\frac s2}\varphi\,dx+\int_{\mathbb{R}^2}|x|^2u\varphi\,dx
-\int_{\mathbb{R}^2}|u|^2u\varphi\,dx-\mu_s\int_{\mathbb{R}^2}u\varphi\,dx=0
\end{equation}
for any $\varphi\in C_c^\infty(\mathbb{R}^2)$, where
\begin{equation}\label{eq:1.18}
\mu_N=\frac{\langle E'_{N,s}(u),u \rangle}{\int_{\mathbb{R}^2}|u|^2\,dx}.
\end{equation}

Motivated by \cite{Bellazzini2016,Yang2020}, we primarily  try to find  local minimizers of $E_{N,s}(u)$ on $S(N)$.
Denote
$$
A_{k}=\{u\in S(N)|\int_{\mathbb{R}^2}|(-\Delta)^{\frac s2}u|^2\,dx\leq k\}.
$$
It seems that not for any $k>0$, the minimization problem $\inf_{u\in A_{k}}E_{N,s}(u)$ is achieved inside $A_{k}$.
After carefully choosing
\begin{equation}\label{eq:1.19}
k =t_s=\frac{N}{2s-1}\Big(\frac{N_s^*}{N}\Big)^{\frac{s}{1-s}},
\end{equation}
where
\begin{equation}\label{eq:1.20}
N_s^*=\int_{\mathbb{R}^2}|Q_s|^2\,dx
\end{equation}
and $Q_s$ is the ground state solution of\eqref{eq:1.13}, we show that
the minimization problem
\begin{equation}\label{eq:1.21}
e_N(s):=\inf_{u\in A_{t_s}}E_{N,s}(u).
\end{equation}
is achieved inside $A_{t_s}$ provide that $N<N^*=\|Q\|_{L^2(\mathbb{R}^2)}^2$ and  $s$ close to $1$.

\bigskip

\begin{Theorem}\label{thm:1}

Suppose $0<N<N^*$. There exists $\varepsilon_N>0$  such that for any $s\in (1-\varepsilon_N,1)$, $E_{N,s}(u)$  admits a nonnegative   minimizer $u_s$ in $A_{t_s}$, that is
\begin{equation}\label{eq:1.22}
E_{N,s}(u_s)=\inf_{u\in A_{t_s}}E_{N,s}(u)\quad ,\quad  E_{N,s}(u_s)<\inf_{u\in \partial A_{t_s}}E_{N,s}(u),
\end{equation}
and $u_s$ is a weak solution of  \eqref{eq:1.12}  with $\mu_s=N^{-1}\langle E'_{N,s}(u_s),u_s \rangle$.
\end{Theorem}

We remark that for a fixed $s\in (\frac 12,1)$, in the same way we can also prove that  the results in Theorem \ref{thm:1}  hold if $N$ is small enough.

If $u\in S(N)$ is a local minimizer of $E_{N,s}(u)$, the inequality $E_{N,s}(u)\geq E_{N,s}(|u|)$ implies that $|u|$ is also a local minimizer. Hence, by Theorem \ref{thm:1}, $E_{N,s}$ has a nonnegative local minimizer.

We note that $t_s\to +\infty$ if $0<N<N^*$  and $t_s\to 0$ if $N>N^*$ whenever $s\to 1_-$. Based on this fact,  we  may show that the local minimizer $u_s\in A_{t_s}$ of $E_{N,s}(u)$ is in fact a ground state, and  for $N>N^*$, $E_{N,s}(u)$  has no local minimizers in $A_{t_s}$ if $s<1$ and $s$ close to 1.  By a ground state $u$ of $E_{N,s}$ on $S(N)$, we mean that $u$ is a critical point of  $E_{N,s}(u)$ on $S(N)$ with the smallest energy among all critical points of  $E_{N,s}(u)$ on $S(N)$.

Let $u_s$   be a nonnegative local minimizer of $E_{N,s}(u)$ inside $A_{t_s}$. We show in the following theorem, among other things,  that any nonnegative local minimizer of $E_{N,s}(u)$ inside $A_{t_s}$ must be radially symmetric.
\begin{Theorem}\label{thm:2}
Suppose $0<N<N^*$. There exists $\varepsilon_N>0$  such that for  $s\in (1-\varepsilon_N,1)$, we have that

$(i)$ $u_s$ is radially symmetric and decreasing with respect to $|x|$;

$(ii)$ $u_s$ is  a ground state of $E_{N,s}(u)$ on $S(N)$;

$(iii)$ The following uniform bounds
\begin{equation}\label{eq:1.23}
\int_{\mathbb{R}^2}|(-\Delta)^{\frac s2}u_s|^2\,dx\leq C
\end{equation}
and
\begin{equation}\label{eq:1.24}
\int_{\mathbb{R}^2}|x|^2|u_s|^2\,dx\leq \frac C2
\end{equation}
hold, where $C>0$ is independent of $s$.
\end{Theorem}

\bigskip

\begin{Theorem}\label{thm:3}
If $N>N^*$, there exists $\varepsilon^1_{N}>0$ such that for any $s\in (1-\varepsilon^1_{N},1)$, $E_{N,s}(u)$  admits no local positive minimizer in $A_{t_s}$.
\end{Theorem}

\bigskip

The functional $E_{N,s}(u)$ actually has a mountain pass geometry, so we may find the second critical point for $E_{N,s}(u)$ by the mountain pass lemma.

\begin{Theorem}\label{thm:4}
Let $0<N<N^*$. There exists $\varepsilon^2_N>0$ such that for  $s\in (1-\varepsilon^2_N,1)$,  $E_{N,s}(u)$  admits a  critical point  $v_s\in \mathcal{H}_s$ at the mountain pass level on $S(N)$.  The critical point $v_s$ is a weak solution of  \eqref{eq:1.12}  with $\mu_s=N^{-1}\langle E'_{N,s}(v_s),v_s \rangle$. Moreover, $v_s$ is radially decreasing with respect to $|x|$.
\end{Theorem}

We observe from theorems \ref{thm:1} and \ref{thm:3} that the  critical particle number $N^*$, which is the threshold of  the classical attractive BEC,  is the same as that for the attractive  non-Gaussian BEC.

\bigskip

Now, we study the asymptotic behavior of the local minimizer $u_s$ and the mountain pass solution $v_s$ as $s\to1_-$. We first show  that $\{u_s\}$  contains a
convergent subsequence as $s\to1_-$.

\begin{Theorem}\label{thm:5}  For any sequence $\{s_k\}$, correspondingly we have $\{u_{s_k}\}$. Suppose that $s_k\to 1_-$ as $k\to \infty$.  Then there exists a subsequence of  $\{s_k\}$, still denoted by $\{s_k\}$, such that $u_{s_k}\rightarrow u$ strongly in $L^2(\mathbb{R}^2)\cap L^4(\mathbb{R}^2)$ as  $k\to \infty$, where $u\in \mathcal{H}$  is a global minimizer of $E(u)$ on $S(N)$.
\end{Theorem}

\bigskip

In contrast to the convergence of a subsequence of $\{u_s\}$,  the mountain pass solutions $\{v_s\}$ are blowing up as $s\to1_-$.
\begin{Theorem}\label{thm:6}

For any sequence $\{s_k\}$, correspondingly we have $\{v_{s_k}\}$. Suppose that $s_k\to 1_-$ as $k\to \infty$.  Then there exists a subsequence of  $\{s_k\}$, still denoted by $\{s_k\}$,  $s_k\to1_-$ as $k\to\infty$, such that
\[
\lim_{k\to\infty}{\|(-\Delta)^{\frac {s_k}{2}}v_{s_k}\|_{L^2(\mathbb{R}^2)}^2}{\Big(\frac{N_{s_k}^*}{N}\Big)^{-\frac{s}{1-s}}}=N.
\]
and
\[
\varepsilon_kv_{s_k}(\varepsilon_kx)\to \frac1{\sqrt{N^*}}Q(\frac x{\sqrt{N}})\quad strongly \quad in \quad L^2(\mathbb{R}^2),
\]
where
\[
\varepsilon_k=\|(-\Delta)^{\frac {s_k}{2}}v_{s_k}\|_{L^2(\mathbb{R}^2)}^{-\frac1{s_k}}
\]
and $Q(x)=Q(|x|)$ is the unique radial positive solution of Eq.\eqref{eq:1.14}.

\end{Theorem}

Results in Theorems \ref{thm:5} and \ref{thm:6} built a connection between mass critical  and supercritical cases of fractional Gross-Pitaevskii equations. A result for the classical Gross-Pitaevskii equation can be found in \cite{Yang2020}. However,  since both the fractional operator and  solutions  in equation \eqref{eq:1.12} depend on  the parameter $s$, it becomes challenging to work out the limiting equation, and delicate estimates are needed in studying the asymptotic behavior of normalized solutions as $s\to 1_-$. In order to overcome the difficulty, we use, among other things, the blowup argument to establish the uniform bound for the solutions, and adopt very weak solutions instead of  weak solutions in the process of passing limit.

We remark that, although we deal with the case $n=2$ and $p=2$ for the sake of physical reasons, our argument can be applied to  the general case  $n\geq 3$ and $p=\frac{4}{n}$.

\bigskip

This paper is   organized as follows. In section 2, we give some preliminary results; In section 3 and 4, we consider the existence and nonexistence of critical points of  $E_{N,s}(u)$, and Theorems \ref{thm:1}--\ref{thm:4} will be proved; then we study the asymptotic behavior of these critical points in section 5, proving Theorems \ref{thm:5}--\ref{thm:6}. In the last section, we present
the similar results of local minimizers for problem \eqref{eq:1.3} with the general potential $V(x)$.

In the following, we utilize $\|\cdot\|_p$ to represent $L^p(\mathbb{R}^2)$ norm.

 \bigskip

\section{Preliminaries}
\bigskip

In this section, we state some results for future reference.

First, we introduce the Pohozaev identity for  the following fractional Gross-Pitaevskii equation
\begin{equation}\label{eq:2.1}
(-\Delta)^su+V(x)u=f(u),\quad x\in \mathbb{R}^n,
\end{equation}
 where $V\in C^1(\mathbb{R}^n,\mathbb{R})$ and $|x\cdot \nabla V|\leq C(1+|V|)$ with $C>0$.
Let
\begin{equation}\label{eq:2.1a}
\mathcal{H}_{s,V}:=\Big\{u\in H^s(\mathbb{R}^2):\int_{\mathbb{R}^2}V(x)|u(x)|^2\,dx<\infty\Big\}.
\end{equation}
\begin{Lemma}\label{lem:2.1}

Assume that $u\in \mathcal{H}_{s,V}$ is a weak solution of \eqref{eq:2.1}.
Then we have  that
\begin{equation}\label{eq:2.2}
s\int_{\mathbb{R}^n}|(-\Delta)^{\frac{s}{2}}u|^2\,dx
-\frac12\int_{\mathbb{R}^n}(x\cdot \nabla V)|u|^2\,dx=\frac n2\int_{\mathbb{R}^n}[f(u)u-2F(u)]\,dx,
\end{equation}
where $F(t)=\int_0^t f(s)\,ds$.
\end{Lemma}
\begin{proof}
It follows from Proposition 4.1 in \cite{Secchi2016} that
 \begin{equation}\label{eq:2.3}
\frac{n-2s}{2}\int_{\mathbb{R}^n}|(-\Delta)^{\frac{s}{2}}u|^2\,dx+\frac n2\int_{\mathbb{R}^n}V(x)|u|^2\,dx
+\frac12\int_{\mathbb{R}^n}(x\cdot \nabla V)|u|^2\,dx=n\int_{\mathbb{R}^n}F(u)\,dx.
\end{equation}
 Multiplying \eqref{eq:2.1} by $u$ and integrating by part, we obtain
\begin{equation}\label{eq:2.4}
\int_{\mathbb{R}^n}|(-\Delta)^{\frac{s}{2}}u|^2\,dx+\int_{\mathbb{R}^n}V(x)|u|^2\,dx=\int_{\mathbb{R}^n}f(u)u\,dx.
\end{equation}
which, together with \eqref{eq:2.3}, yields \eqref{eq:2.2}.
\end{proof}

\bigskip

In order to prove the existence of local minimizers of $E_{N,s}$ in $S(N)$, we need the following fractional Gagliardo-Nirenberg  inequality
\begin{equation}\label{eq:2.5}
\int_{\mathbb{R}^n}|u|^p\,dx\leq C_0\left(\int_{\mathbb{R}^n}|(-\Delta)^{\frac s2}u|^2\,dx\right)^{\frac{n(p-2)}{4s}}
\left(\int_{\mathbb{R}^n}|u|^2\,dx\right)^{\frac p2-\frac{n(p-2)}{4s}},
\end{equation}
where $p\in[2,\frac{2n}{n-2s})$. In particular, if $n=2$ and $p=4$, we have
\begin{equation}\label{eq:2.6}
\int_{\mathbb{R}^2}|u|^4\,dx\leq C_0\left(\int_{\mathbb{R}^2}|(-\Delta)^{\frac s2}u|^2\,dx\right)^{\frac{1}{s}}
\left(\int_{\mathbb{R}^2}|u|^2\,dx\right)^{2-\frac{1}{s}},
\end{equation}
where $u\in H^{\frac{s}{2}}(\mathbb{R}^2)$ and $\frac{1}{2}<s\leq 1$. It is known from \cite{Frank-Lenzmann2016}, or Lemma 2.1 in \cite{He-Long2016},
that the maximizing problem
\begin{equation}\label{eq:2.7}
C_0= \max_{u\in H^{\frac{s}{2}}(\mathbb{R}^2)}\frac{\int_{\mathbb{R}^2}|u|^4\,dx}{\left(\int_{\mathbb{R}^2}|(-\Delta)^{\frac s2}u|^2\,dx\right)^{\frac{1}{s}}
\left(\int_{\mathbb{R}^2}|u|^2\,dx\right)^{2-\frac{1}{s}}},
\end{equation}
is achieved, and
\[
 C_0=\frac{2s}{(2s-1)^{1-1/s}\|Q_s\|_2^2},
\]
where $Q_s\in H^s(\mathbb{R}^2)$,  $Q_s\geq0$ and $Q_s\neq0$, is a ground state solution of  \eqref{eq:1.13}.

\bigskip

By Proposition 3.1, Lemmas 8.2 and 8.4 of \cite{Frank-Lenzmann2016}, we know further properties of $Q_s$:

\begin{Lemma}\label{lem:2.3}
$(i)$ There exists $x_0\in \mathbb{R}^2$ such that $Q_s(\cdot-x_0)$ is radial, positive and strictly decreasing in $|x-x_0|$. Furthermore,
$Q_s\in H^{2s+1}(\mathbb{R}^2)\cap C^\infty(\mathbb{R}^2)$ and
\[
\frac{C_1}{1+|x|^{N+2s}}\leq Q(x)\leq \frac{C_2}{1+|x|^{N+2s}}\quad for \quad x\in \mathbb{R}^2,
\]
where $C_2\geq C_1>0$ may be dependent of $s$;

$(ii)$ for any $s\in (\frac12,1)$, it holds that
\[
\int_{\mathbb{R}^2}Q_s\,dx\leq C_3,\quad Q_s\leq C_4|x|^{-2} \quad for \quad any \quad |x|\geq 1,
\]
where $C_3>0$, $C_4>0$ are  independent of $s$.
\end{Lemma}

\bigskip

Theorem 4 in \cite{Frank-Lenzmann2016} proved the following uniqueness for the ground state solution $Q_s$.
\begin{Lemma}\label{lem:2.4}
The ground state solution $Q_s\in H^s(\mathbb{R}^2)$ of equation \eqref{eq:1.13} is unique up to translation.
\end{Lemma}

\bigskip

Let  $Q_s(x)\in H^{\frac{s}{2}}(\mathbb{R}^2)$ be the unique radial positive ground state of  \eqref{eq:1.13}, and $Q$ be the unique positive solution of \eqref{eq:1.14}.  Then, we have that
\begin{Lemma}\label{lem:2.5} For $\frac12<s\leq1$, there holds that
\begin{equation}\label{eq:2.8}
\int_{\mathbb{R}^2}|(-\Delta)^{\frac{s}{2}}Q_s|^2\,dx
=\frac{1}{2s}\int_{\mathbb{R}^2}|Q_s|^4\,dx=\frac{1}{2s-1}\int_{\mathbb{R}^2}|Q_s|^2\,dx.
\end{equation}
Moreover, $Q_s\to Q$ strongly in $L^2(\mathbb{R}^2)$ as $s\to1_-$, and then
\[
N_s^*:=\|Q_s\|_2^2\to N^*=\|Q\|_2^2
\]
as $s\to1_-.$
\end{Lemma}

\begin{proof}
Taking $V(x)=1$, $n=2$ and $f(u)=|u|^2u$ in Lemma \ref{lem:2.1}, we find
\begin{equation}\label{eq:2.9}
(1-s)\int_{\mathbb{R}^2}|(-\Delta)^{\frac{s}{2}}Q_s|^2\,dx+\int_{\mathbb{R}^2}|Q_s|^2\,dx
=\frac12\int_{\mathbb{R}^2}|Q_s|^4\,dx.
\end{equation}
Since $Q_s$ is a solution of  \eqref{eq:1.13}, we get
\begin{equation}\label{eq:2.10}
\int_{\mathbb{R}^2}|(-\Delta)^{\frac{s}{2}}Q_s|^2\,dx+\int_{\mathbb{R}^2}|Q_s|^2\,dx
=\int_{\mathbb{R}^2}|Q_s|^4\,dx.
\end{equation}
Equation \eqref{eq:2.8} follows from  \eqref{eq:2.9} and \eqref{eq:2.10}. The convergence $Q_s\to Q$ strongly in $L^2(\mathbb{R}^2)$ as $s\to1_-$ is a consequence of Lemmas 8.3 and 8.4 of \cite{Frank-Lenzmann2016}.
\end{proof}

Now, we derive a Br\'{e}zis-Kato type inequality.
\begin{Lemma}\label{lem:2.6}
For $u\in H^s(\mathbb{R}^n)$, $0<s<1$, we have that
\begin{equation}\label{eq:2.11}
\int_{\mathbb{R}^n}|(-\Delta)^{\frac s2}|u||^2\,dx\leq \int_{\mathbb{R}^n}|(-\Delta)^{\frac s2}u|^2\,dx.
\end{equation}
\end{Lemma}
\begin{proof}
By Lemma 3.1 in \cite{Frank2008}, for any $u\in H^s(\mathbb{R}^n)$,
\[
\int_{\mathbb{R}^n}|(-\Delta)^{\frac s2}u|^2\,dx=\int_{\mathbb{R}^n}|\xi|^{2s}|\hat{u}|^2\,d\xi
=\int_{\mathbb{R}^n}\int_{\mathbb{R}^n}\frac{(u(x)-u(y))^2}{|x-y|^{n+2s}}\,dxdy.
\]
For any given $u\in H^1(\mathbb{R}^n)$, we know that $|u|\in H^1(\mathbb{R}^n)\subset H^s(\mathbb{R}^n)$. Hence,
\begin{equation*}
\begin{split}
\int_{\mathbb{R}^n}|(-\Delta)^{\frac s2}|u||^2\,dx&=\int_{\mathbb{R}^n}\int_{\mathbb{R}^n}\frac{(|u(x)|-|u(y)|)^2}{|x-y|^{n+2s}}\,dxdy\\
&\leq\int_{\mathbb{R}^n}\int_{\mathbb{R}^n}\frac{(u(x)-u(y))^2}{|x-y|^{n+2s}}\,dxdy\\
&=\int_{\mathbb{R}^n}|(-\Delta)^{\frac s2}u|^2\,dx.
\end{split}
\end{equation*}
The conclusion follows by the density argument.
\end{proof}

\bigskip

We have the following compact result.
\begin{Lemma}\label{lem:2.7} The inclusion $\mathcal{H}_s\hookrightarrow L^p(\mathbb{R}^2)$ is compact for $ p\in[2,\frac 2{1-s})$.
\end{Lemma}
\begin{proof}
Let $\{u_n\}$ be a bounded sequence in $\mathcal{H}_s$. We may assume that
\[
u_n\rightharpoonup u \quad {\rm weakly\quad in}  \quad \mathcal{H}_s;\quad
u_n\to u \quad {\rm strongly \quad in}\quad L_{loc}^p(\mathbb{R}^2)\quad  {\rm for}\quad p\in[2,\frac 2{1-s}).
\]
For any $\varepsilon>0$, there exists $R$ large enough such that
\begin{equation}\label{eq:2.12}
\int_{\{|x|\geq R\}}|u|^2\,dx<\frac\varepsilon8
\end{equation}
and
\begin{equation}\label{eq:2.13}
\begin{split}
\int_{\{|x|\geq R\}}|u_n|^2\,dx\leq \frac1{R^2}\int_{\{|x|\geq R\}}|x|^2|u_n|^2\,dx\leq \frac1{R^2}\int_{\mathbb{R}^2}|x|^2|u_n|^2\,dx<\frac\varepsilon8.
\end{split}
\end{equation}
Since $u_n\to u$ strongly in $L_{loc}^2(\mathbb{R}^2)$,  there exists $N\in \mathbb{N}$ such that if $n>N$,
\begin{equation}\label{eq:2.14}
\int_{\{|x|\leq R\}}|u_n-u|^2\,dx<\frac\varepsilon2.
\end{equation}
We deduce from  \eqref{eq:2.12}, \eqref{eq:2.13} and  \eqref{eq:2.14} that
\begin{equation*}
\begin{split}
\int_{\mathbb{R}^2}|u_n-u|^2\,dx&= \big(\int_{\{|x|\leq R\}}+\int_{\{|x|\geq R\}}\big)|u_n-u|^2\,dx\\
&<\frac\varepsilon2+2\big(\int_{\{|x|\geq R\}}|u_n|^2\,dx+\int_{\{|x|\geq R\}}|u|^2\,dx\big)\\
&<\frac\varepsilon2+2\big(\frac\varepsilon8+\frac\varepsilon8\big)=\varepsilon.
\end{split}
\end{equation*}
Thus $u_n\to u$ strongly in $L^2(\mathbb{R}^2)$.  By the fractional Gagliardo-Nirenberg  inequality \eqref{eq:2.5}, $u_n\to u$ strongly in $L^p(\mathbb{R}^2)$  for any $ p\in[2,\frac 2{1-s})$. The proof is complete.

\end{proof}

\bigskip

Finally, we recall a result stated in Theorem 4 in \cite{Bellazzini2017}.

\begin{Lemma}\label{lem:2.8} Let $V:\mathbb{R}^n\to [0,\infty)$  be a measurable function, radially symmetric satisfying
$V(|x|)\leq V(|y|)$ for $|x|\leq |y|$  then we have:
\begin{equation}\label{eq:2.15}
\int_{\mathbb{R}^n}V(|x|)|u^*|^2\,dx \leq \int_{\mathbb{R}^n}V(|x|)|u|^2\,dx
\end{equation}
If in addition $V(|x|)< V(|y|)$ for $|x|<|y|$ then
\begin{equation}\label{eq:2.16}
\int_{\mathbb{R}^n}V(|x|)|u^*|^2\,dx =
\int_{\mathbb{R}^n}V(|x|)|u|^2\,dx\Rightarrow u(x) = u^*(|x|).
\end{equation}
\end{Lemma}

\bigskip

\section{Local minimizers}

\bigskip

In this section, we study the existence and nonexistence of local minimizers of $E_{N,s}(u)$ in  $A_{t_s}$. We show that $E_{N,s}(u)$ has a local minimizer inside $A_{t_s}$ if $0<N<N^*=\|Q\|_{L^2(\mathbb{R}^2)}^2$ and $s$ close to $1$, while $E_{N,s}(u)$ does not admit a minimizer in $A_{t_s}$ if $N>N^*$. Moreover, the minimizer is actually a ground state of $E_{N,s}(u)$ if $s$ close to $1$.

We first show that local minimizers of $E_{N,s}(u)$ in  $A_{t_s}$ can not achieved on the boundary of   $A_{t_s}$.
\begin{Lemma}\label{lem:3.1}
Let $0<N<N^*$. There exists $\varepsilon_N>0$ for any $s\in (1-\varepsilon_N,1)$ such that
\begin{equation}\label{eq:3.1}
e_N(s)=\inf_{u\in A_{t_s}}E_{N,s}(u)\leq C<\inf_{u\in \partial A_{t_s}}E_{N,s}(u),
\end{equation}
where $C>0$ is independent of $s$.
\end{Lemma}

\begin{proof}
Firstly, we find a proper $t_s>0$.

For $u\in S(N)$, by the Gagliardo-Nirenberg inequality \eqref{eq:2.6}, we have
\begin{equation}\label{eq:3.2}
\begin{split}
E_{N,s}(u)\geq\frac{1}{2}\int_{\mathbb{R}^2}\big(|(-\Delta)^{\frac s2} u|^2+|x|^2|u|^2\big)\,dx-\frac{1}{4}C_0
 \left(\int_{\mathbb{R}^2}|(-\Delta)^{\frac s2}u|^2\,dx\right)^{\frac{1}{s}}
\left(\int_{\mathbb{R}^2}|u|^2\,dx\right)^{2-\frac{1}{s}}
\end{split}
\end{equation}

For $s>\frac12$, let $\lambda_s=\frac s{(2s-1)^{1-\frac1s}}$ and
\[
f(t)=t-\lambda_s\frac{N^{2-\frac1s}}{N_s*}t^{\frac1s}.
\]
Let $t_s$ be defined in \eqref{eq:1.19}. We may verify that $f(t)$ is increasing in $[0,t_s]$ and  decreasing in $[t_s,+\infty)$. Hence,
\begin{equation*}
\max_{t>0}f(t)=f(t_s)=(1-s)t_s=\frac{N(1-s)}{2s-1}\Big(\frac{N_s^*}{N}\Big)^{\frac{s}{1-s}}.
\end{equation*}
It follows from \eqref{eq:3.2} that
\begin{equation}\label{eq:3.3}
\inf_{u\in \partial A_{t_s}}E_{N,s}(u)\geq \frac12f(t_s)=\frac12(1-s)t_s.
\end{equation}

\bigskip

Secondly, we bound $e_N(s)$ by a positive constant $C$, which is independent of $s$.

Choose $\varphi\in C_c^\infty(\mathbb{R}^2)$ such that $\|\varphi\|_2^2=N$. By the H\"{o}lder inequality, we get
\begin{equation}\label{eq:3.4}
\begin{split}
\int_{\mathbb{R}^2}|(-\Delta)^{\frac s2}\varphi|^2\,dx&=\int_{\mathbb{R}^2}|\xi|^{2s}|\widehat{\varphi}|^2\,d\xi\\
&\leq \Big(\int_{\mathbb{R}^2}|\xi|^{2}|\widehat{\varphi}|^2\,dx\Big)^{s}
\Big(\int_{\mathbb{R}^2}|\widehat{\varphi}|^2\,dx\Big)^{1-s}\\
&\leq \Big(1+\int_{\mathbb{R}^2}|\xi|^{2}|\widehat{\varphi}|^2\,dx\Big)(1+N).
\end{split}
\end{equation}
Therefore, there exists $C>0$ independent of $s$ such that $E_{N,s}(\varphi)\leq C$.

Since $0<N<N^*$ and by Lemma \ref{lem:2.5}, $N_s^*\to N^*$, there exists $\varepsilon_N>0$ such that if $s\in (1-\varepsilon_N,1)$,
\begin{equation}\label{eq:3.5}
\frac12(1-s)t_s=\frac{2N(1-s)}{2s-1}\Big(\frac{N_s^*}{N}\Big)^{\frac{s}{1-s}}>C
\end{equation}
and
\begin{equation}\label{eq:3.6}
\Big(1+\int_{\mathbb{R}^2}|\xi|^{2}|\widehat{\varphi}|^2\,dx\Big)(1+N)<t_s.
\end{equation}
It follows from \eqref{eq:3.3} and \eqref{eq:3.6} that $\varphi\in A_{t_s}$. Thus, \eqref{eq:3.3} and \eqref{eq:3.5} yield that  for $s\in (1-\varepsilon_N,1)$,
\begin{equation*}
e_N(s)=\inf_{u\in A_{t_s}}E_{N,s}(u)\leq E_{N,s}(\varphi)\leq C<\inf_{u\in \partial A_{t_s}}E_{N,s}(u).
\end{equation*}
\end{proof}

\bigskip

\begin{Lemma}\label{lem:3.2}
Suppose $s\in (\frac 12, 1)$. Let $\{u_n\}\subseteq A_{t_s}$ be the minimizing sequence of $e_N(s)$. That is,
\[
E_{N,s}(u_n)\to e_N(s)<\inf_{u\in \partial A_{t_s}}E_{N,s}(u) \quad {\rm as} \quad n\to\infty.
\]
 Then, there exists a subsequence of $\{u_n\}$ converging strongly to a local minimizer $u$ of  $e_N(s)$.
\end{Lemma}

\begin{proof}
The assumption $E_{N,s}(u_n)\to e_N(s)$ as $n\to\infty$ implies that
\begin{equation}\label{eq:3.7}
E_{N,s}(u_n)\leq 2e_N(s)
\end{equation}
provided that $n$ is  large enough. Since $\{u_n\}\subseteq A_{t_s}$, that is,
\begin{equation*}\label{unb}
\int_{\mathbb{R}^2}|(-\Delta)^{\frac s2} u_n|^2\,dx\leq t_s,
\end{equation*}
we deduce from  \eqref{eq:3.2} and \eqref{eq:3.7} that
\begin{equation*}\label{unb2}
\int_{\mathbb{R}^2}|x|^2 |u_n|^2\,dx\leq 4e_N(s)+\frac{1}{2}C_0t_s^{\frac1s}N^{2-\frac{1}{s}}.
\end{equation*}
Hence, $\{u_n\}$ is uniformly bounded in $\mathcal{H}_s$. We may assume that for $s>\frac12$,
$u_n\rightharpoonup u$ weakly in $\mathcal{H}_s$;
$u_n\to u$ strongly in $L_{loc}^p(\mathbb{R}^2)$ for $p\in[2,\frac 2{1-s})$ as $n\to\infty$.  Moreover,
\begin{equation}\label{eq:3.8}
\|u\|_{\mathcal{H}_s}\leq \liminf_{n\to\infty}\|u_n\|_{\mathcal{H}_s}.
\end{equation}

On the other hand, by Lemma \ref{lem:2.7}, $u_n\to u$ strongly in $L^2(\mathbb{R}^2)\cap L^4(\mathbb{R}^2)$ and $u\in A_{t_s}$.
The fact $E_{N,s}(u_n)\to e_N(s)$ as $n\to\infty$  and $E_{N,s}(u)\geq e_N(s)$ yield that
\begin{equation}\label{eq:3.9}
\begin{split}
\lim_{n\to\infty}\|u_n\|_{\mathcal{H}_s}^2
&=\lim_{n\to\infty}\Big[2E_{N,s}(u_n)+\frac12\int_{\mathbb{R}^2}|u_n|^4\,dx\Big]\\
&=2e_N(s)+\frac12\int_{\mathbb{R}^2}|u|^4\,dx\\
&\leq 2E_{N,s}(u)+\frac12\int_{\mathbb{R}^2}|u|^4\,dx\\
&=\|u\|_{\mathcal{H}_s}^2.
\end{split}
\end{equation}
It follows  from \eqref{eq:3.8} and \eqref{eq:3.9} that $\lim_{n\to\infty}\|u_n\|_{\mathcal{H}_s}=\|u\|_{\mathcal{H}_s}$, thereby
$u_n\to u$ strongly in $\mathcal{H}_s$ as $n\to\infty$. So we have that $E_{N,s}(u_n)\to E_{N,s}(u)$ as $n\to\infty$. This implies  that $E_{N,s}(u)=e_N(s)$,
namely,  $u$ is a local minimizer of $E_{N,s}$ in $A_{t_s}$.
\end{proof}

\bigskip

Now, we are in position to prove Theorem \ref{thm:1}.\\
{\bf Proof of Theorem \ref{thm:1}.} By Lemmas \ref{lem:3.1} and \ref{lem:3.2},
there exists $\varepsilon_N>0$  such that for any $s\in (1-\varepsilon_N,1)$,  $E_{N,s}(u)$  admits a local  minimizer $u_s$ inside $A_{t_s}$. Hence, $u_s$
is a solution of \eqref{eq:1.12} for  $\mu_s\in \mathbb{R}$, where $\mu_s$  can be determined by \eqref{eq:1.12}.
Furthermore, in terms of Lemma \ref{lem:2.6}, $E_{N,s}(|u_s|)\leq E_{N,s}(u_s)$. Thus, we may assume that  $u_s$ is nonnegative. The proof is complete.
\quad $\Box$

\bigskip

In order to bound uniformly the local  minimizer $u_s$ in $s$ and show that it is  a ground state solution of \eqref{eq:1.12}, we need the following result.\\
\begin{Lemma}\label{lem:3.3}
For any $\varphi\in C_c^\infty(\mathbb{R}^n)$ with $n>1$, it holds that

$(i)$ $\lim_{s\to1_-}(-\Delta)^s\varphi=-\Delta \varphi$;

$(ii)$ As $s\to1_-$, $(-\Delta)^s \varphi \to -\Delta \varphi$ strongly in $L^2(\mathbb{R}^n)$.
\end{Lemma}

\begin{proof}
The conclusion $(i)$ follows from Proposition 4.4 in \cite{Nezza2012}.

Now, we prove $(ii)$. For any $\psi\in L^2(\mathbb{R}^2)$ and $\varphi\in C_c^\infty(\mathbb{R}^n)$,  we derive from the
Young inequality that for $s\in(0,1)$,
\begin{equation}\label{eq:3.10}
|\xi|^s|\hat{\varphi}\hat{\psi}|\leq (1+|\xi|^2)|\hat{\varphi}\hat{\psi}|\leq \frac12[(1+|\xi|^2)^2|\hat{\varphi}|^2+|\hat{\psi}|^2]\in L^1(\mathbb{R}^n).
\end{equation}
Hence,  the dominated convergence theorem yields for $\psi\in L^2(\mathbb{R}^2)$ and $s\to 1_-$ that
\begin{equation}\label{eq:3.11}
\int_{\mathbb{R}^n}(-\Delta)^{s}\varphi\psi\,dx=\int_{\mathbb{R}^n}|\xi|^{2s}\hat{\varphi}\hat{\psi}\,dx
\to\int_{\mathbb{R}^n}|\xi|^2\hat{\varphi}\hat{\psi}\,dx=\int_{\mathbb{R}^n}(-\Delta\varphi)\psi\,dx.
\end{equation}
This means that $(-\Delta)^s \varphi \rightharpoonup -\Delta \varphi$ weakly in $L^2(\mathbb{R}^2)$, as $s\to1_-$.
Similarly, we can deduce that
\[
\lim_{s\to 1_-}\int_{\mathbb{R}^n}|(-\Delta)^{s}\varphi|^2\,dx=\int_{\mathbb{R}^n}|(-\Delta\varphi)|^2\,dx.
\]
Therefore, $(-\Delta)^s \varphi \to -\Delta \varphi$ strongly in $L^2(\mathbb{R}^n)$, as $s\to1_-$.
\end{proof}

\bigskip

{\bf Proof of Theorem \ref{thm:2}.}  Let  $u_s\in A_{t_s}$ be the local minimizer of $E_{N,s}$ obtained by Theorem \ref{thm:1}.

We first prove that $u_s$ is radially decreasing with respect to $x$. Denote by $u_s^*$  the Schwartz rearrangement  of $u_s$ with respect to $x$. By   Theorem A.1 in \cite{Frank-Seiringer2008},
\[
\int_{\mathbb{R}^2}|(-\Delta )^{\frac s2}u_s^*|^2\,dx\leq \int_{\mathbb{R}^2}|(-\Delta )^{\frac s2}u_s|^2\,dx
\]
and for $p\geq 2$,
\[
\int_{\mathbb{R}^2}|u_s^*|^p\,dx= \int_{\mathbb{R}^2}|u_s|^p\,dx.
\]
Lemma \ref{lem:2.8} yields that
\[
\int_{\mathbb{R}^2}|x|^2|u_s^*|^2\,dx\leq  \int_{\mathbb{R}^2}|x|^2|u_s|^2\,dx.
\]
Therefore,
\begin{equation}\label{eq:3.12}
E_{N,s}(u_s^*)\leq E_{N,s}(u_s).
\end{equation}
Since $u_s$ is a local minimizer of $e_s(N)$, we obtain that
\[
\int_{\mathbb{R}^2}|(-\Delta )^{\frac s2}u_s^*|^2\,dx+\int_{\mathbb{R}^2}|x|^2|u_s^*|^2\,dx= \int_{\mathbb{R}^2}|(-\Delta )^{\frac s2}u_s|^2\,dx+\int_{\mathbb{R}^2}|x|^2|u_s|^2\,dx.
\]
This implies that
\[
\int_{\mathbb{R}^2}|x|^2|u_s^*|^2\,dx=\int_{\mathbb{R}^2}|x|^2|u_s|^2\,dx.
\]
By Lemma \ref{lem:2.8}, we obtain $u_s=u_s^*$. This means that $u_s$ is radially decreasing with respect to $|x|$.

\bigskip

Next, we prove that the energy $E_{N,s}(u_s)$  is the smallest one among all critical points of $E_{N,s}(u)$ . It suffices to prove that  $E_{N,s}(v_s)= E_{N,s}(u_s)$ for any critical point $v_s$ satisfying $E_{N,s}(v_s)\leq E_{N,s}(u_s)$.

Since $v_s$ is a critical point of $E_{N,s}$ on $S(N)$, for any $\varphi\in \mathcal{H}_s$, $v_s$ satisfies \eqref{eq:1.17}, that is,
\begin{equation*}
\int_{\mathbb{R}^2}(-\Delta)^{\frac{s}2}v_s(-\Delta)^{\frac s2}\varphi\,dx+\int_{\mathbb{R}^2}|x|^2v_s\varphi\,dx
-\int_{\mathbb{R}^2}|v_s|^2v_s\varphi\,dx-\mu_s\int_{\mathbb{R}^2}v_s\varphi\,dx=0,
\end{equation*}
where
\[
\mu_s= \frac 1N\int_{\mathbb{R}^2}(|(-\Delta)^{\frac{s}2}v_s|^2 + x|^2|v_s|^2- |v_s|^4)\,dx.
\]
By Lemma \ref{lem:2.1}, we have
\begin{equation}\label{eq:3.13}
s\int_{\mathbb{R}^2}|(-\Delta)^{\frac{s}{2}}v_s|^2\,dx-\int_{\mathbb{R}^2}|x|^2|v_s|^2\,dx
-\frac12\int_{\mathbb{R}^2}|v_s|^4\,dx=0.
\end{equation}
Thus, by Lemma \ref{lem:2.4} and \eqref{eq:3.1} we infer that
\begin{equation}\label{eq:3.14a}
\begin{split}
2E_{N,s}(v_s)&=\int_{\mathbb{R}^2}\big(|(-\Delta)^{\frac s2}v_s|^2+|x|^2|v_s|^2\big)\,dx-\frac12\int_{\mathbb{R}^2}|v_s|^4\,dx\\
&=(1-s)\int_{\mathbb{R}^2}|(-\Delta)^{\frac s2}v_s|^2\,dx+2\int_{\mathbb{R}^2}|x|^2|v_s|^2\,dx\\
&\leq 2E_{N,s}(u_s)\leq C,
\end{split}
\end{equation}
and then
\begin{equation}\label{eq:3.14}
\int_{\mathbb{R}^2}|(-\Delta)^{\frac s2}v_s|^2\,dx\leq \frac C{1-s}<\frac12t_s
\end{equation}
as well as
\begin{equation}\label{eq:3.15}
\int_{\mathbb{R}^2}|x|^2|v_s|^2\,dx\leq \frac C2.
\end{equation}
 Therefore,   $v_s\in A_{t_s}$ implying
\[
E_{N,s}(v_s)\geq \inf_{u\in A_{t_s}}E_{N,s}(u)=E_{N,s}(u_s).
\]
This and the assumption $E_{N,s}(v_s)\leq E_{N,s}(u_s)$ yield $E_{N,s}(v_s)= E_{N,s}(u_s)$.

\bigskip

Finally, we build the bound in $(iii)$. We remark that $u_s$ also satisfies \eqref{eq:3.13}--\eqref{eq:3.15}. It is then sufficient to show that $\int_{\mathbb{R}^2}|(-\Delta)^{\frac s2}u_s|^2\,dx$ is uniformly bounded with respect to $s$ for $s$ close to $1$.

Were it not the case, there would exist $\{s_k\}$,  $s_k\to1_-$, such that
\begin{equation}\label{eq:3.16}
\int_{\mathbb{R}^2}|(-\Delta)^{\frac {s_k}2}u_{s_k}|^2\,dx\to+\infty
\end{equation}
as $k\to\infty$. Therefore, for some $K\in \mathbb{N}$, if $k>K$, we have
\begin{equation}\label{eq:3.17}
\int_{\mathbb{R}^2}|(-\Delta)^{\frac {s_k}2}u_{s_k}|^2\,dx>1.
\end{equation}
For simplicity, denote $u_k=u_{s_k}$. It follows  from \eqref{eq:3.15}, \eqref{eq:3.16} and   Lemma \ref{lem:2.1} that
\begin{equation}\label{eq:3.18}\lim_{k\to\infty}\frac{\int_{\mathbb{R}^2}|(-\Delta)^{\frac s2}u_k|^2\,dx}{\int_{\mathbb{R}^2}|u_k|^4\,dx}=\frac12.
\end{equation}
 Set $w_k=\eta_ku_k(\eta_kx)$ with $\eta_k=\|(-\Delta)^{\frac {s_k}2}u_k\|_2^{-\frac1{s_k}}$. We have that
\begin{equation}\label{eq:3.19}
\int_{\mathbb{R}^2}|(-\Delta)^{\frac {s_k}2}w_k|^2\,dx=\eta_k^{2s_k}\int_{\mathbb{R}^2}|(-\Delta)^{\frac {s_k}2}u_k|^2\,dx=1
\end{equation}
and
\[
\int_{\mathbb{R}^2}|w_k|^2\,dx=N.
\]
By \eqref{eq:3.17},  for $k>K$ there holds that
\begin{equation}\label{eq:3.20}
\eta_k^{2s_k-2}=\Big(\int_{\mathbb{R}^2}|(-\Delta)^{\frac s2}u_{s_k}|^2\,dx\Big)^{\frac{1-s_k}{s_k}}>1.
\end{equation}
However, \eqref{eq:3.14} yields that
\begin{equation}\label{eq:3.21}
\eta_k^{2s_k-2}=\Big(\int_{\mathbb{R}^2}|(-\Delta)^{\frac s2}u_{s_k}|^2\,dx\Big)^{\frac{1-s_k}{s_k}}\leq
\Big(\frac C{1-s_k}\Big)^{\frac{1-s_k}{s_k}}\to 1,\quad {\rm as} \quad k\to\infty.
\end{equation}
As a result,
\begin{equation}\label{eq:3.22}
\lim_{k\to\infty}\eta_k^{2s_k-2}=1.
\end{equation}
Therefore,
\begin{equation}\label{eq:3.23}
\begin{split}
\int_{\mathbb{R}^2}|w_k|^4\,dx&=\eta_k^2\int_{\mathbb{R}^2}|u_k|^4\,dx\\
&=\frac{\int_{\mathbb{R}^2}|u_k|^4\,dx}{\int_{\mathbb{R}^2}|(-\Delta)^{\frac s2}u_k|^2\,dx}\eta_k^{\frac{2(1-s_k)}{s_k}}\to 2
\end{split}
\end{equation}
 as $k\to\infty$.
By \eqref{eq:3.19} and the H\"{o}lder inequality, for $s_k>\frac34$ there holds
\begin{equation}\label{eq:3.24}
\begin{split}
\int_{\mathbb{R}^2}|(-\Delta)^{\frac {3}8}w_k|^2\,dx&=\int_{\mathbb{R}^2}(|\xi|^{\frac34}|\mathcal{F}w_k|^2\,dx\\
&\leq
\Big(\int_{\mathbb{R}^2}(|\xi|^{s_k}|\mathcal{F}w_k|^2\,dx\Big)^{\frac{3}{4s_k}}
\Big(\int_{\mathbb{R}^2}(\mathcal{F}w_k|^2\,dx\Big)^{1-\frac{3}{4s_k}}\\
&=
\Big(\int_{\mathbb{R}^2}|(-\Delta)^{\frac {s_k}2}w_k|^2\,dx\Big)^{\frac{3}{4s_k}}
\Big(\int_{\mathbb{R}^2}(|w_k|^2\,dx\Big)^{1-\frac{3}{4s_k}}\\
&=\Big(\int_{\mathbb{R}^2}(|w_k|^2\,dx\Big)^{1-\frac{3}{4s_k}}\\
&\leq (1+N)^{1-\frac{3}{4s_k}}\leq 1+N.
\end{split}
\end{equation}
By \eqref{eq:3.15},
\begin{equation}\label{eq:3.25}
\int_{\mathbb{R}^2}|x|^2|w_k|^2\,dx=\eta_k^{-2}\int_{\mathbb{R}^2}|x|^2|u_k|^2\,dx\leq C\eta_k^{-2}.
\end{equation}
Hence, we may assume that  $w_k\rightharpoonup w$ weakly in $H^{\frac34}(\mathbb{R}^2)$ and $w_k\to w$ strongly in $L_{loc}^p(\mathbb{R}^2)$ for any $2\leq p<8$ as $k\to\infty$. Apparently,
\begin{equation}\label{eq:3.26}
\int_{\mathbb{R}^2}|w|^2\,dx\leq \liminf_{k\to\infty}\int_{\mathbb{R}^2}|w_k|^2\,dx=N.
\end{equation}
Since the inclusion $H_r^{\frac34}(\mathbb{R}^2)\hookrightarrow L^p(\mathbb{R}^2)$ is compact, see Theorem $\Pi$.1 in \cite{Lions1982}, \eqref{eq:3.23} gives
us that
\[
\int_{\mathbb{R}^2}|w|^4\,dx=2.
\]
Moreover, we deduce from
\begin{equation}\label{eq:3.27}
(-\Delta)^{s_k}u_k+|x|^2u_k=|u_k|^2u_k+\mu_ku_k
\end{equation}
that
\begin{equation}\label{eq:3.28}
(-\Delta)^{s_k}w_k+\eta_k^{2s_k}|x|^2w_k=\eta_k^{2s_k-2}|w_k|^2w_k+\mu_k\eta_k^{2s_k}w_k.
\end{equation}
Using the Plancherel  theorem, we drive that $w_k\in \mathcal{H}_s$ is also a very
weak solution of  \eqref{eq:3.28}.
Namely, for any $\varphi\in C_c^\infty(\mathbb{R}^2)$, we have that
\begin{equation}\label{eq:3.29}
\int_{\mathbb{R}^2}w_k(-\Delta)^{s_k}\varphi\,dx+\eta_k^{2s_k}\int_{\mathbb{R}^2}|x|^2w_k\varphi\,dx
=\eta_k^{2s_k-2}\int_{\mathbb{R}^2}|w_k|^2w_k\varphi\,dx+\mu_k\eta_k^{2s_k}\int_{\mathbb{R}^2}w_k\varphi\,dx.
\end{equation}
Since
\begin{equation}\label{eq:3.30}
\mu_kN=\int_{\mathbb{R}^2}|(-\Delta)^{\frac {s_k}2}u_k|^2\,dx+\int_{\mathbb{R}^2}|x|^2|u_k|^2\,dx-\int_{\mathbb{R}^2}|u_k|^4\,dx,
\end{equation}
we get
\begin{equation}\label{eq:3.31}
\begin{split}
\mu_k\eta_k^{2s_k}N&=\mu_kN\|(-\Delta)^{\frac s2}u_k\|_2^{-2}\\
&=1+\frac{\int_{\mathbb{R}^2}|x|^2|u_k|^2\,dx}{\|(-\Delta)^{\frac s2}u_k\|_2^{2}}-\frac{\int_{\mathbb{R}^2}|u_k|^4\,dx}{\|(-\Delta)^{\frac s2}u_k\|_2^{2}}\\
&\to -1,\quad as \quad k\to\infty.
\end{split}
\end{equation}
Whence by  Lemma \ref{lem:3.3}, \eqref{eq:3.22} and \eqref{eq:3.31}, we see that $w$ is a very weak solution to
\begin{equation*}\label{lime}
-\Delta w+\frac1Nw=|w|^2w.
\end{equation*}
It is standard to verify that $w\in H^1(\mathbb{R}^2)$.
The scaled function $v=(N)^{-\frac12}w( N^{-\frac12}x)$ then satisfies
\begin{equation*}\label{lime1}
-\Delta v+v=|v|^2v.
\end{equation*}
By  the Gagliardo-Nirenberg inequality, \eqref{eq:3.26} and Lemma \ref{lem:2.5}, we obtain
\begin{equation*}
\frac12N^*\leq\frac{\int_{\mathbb{R}^2}|\nabla v|^2\,dx\int_{\mathbb{R}^2}|v|^2\,dx}{\int_{\mathbb{R}^2}| v|^4\,dx}=
\frac12\int_{\mathbb{R}^2}|v|^2\,dx=\frac12\int_{\mathbb{R}^2}|w|^2\,dx\leq\frac12N<\frac12N^*
\end{equation*}
a contradiction, which implies that our  assertion is valid. The proof is complete.\quad$\Box$

\bigskip

Finally, we prove that there does not exist  local  minimizers   of $E_{N,s}(u)$ in $A_{t_s}$ for $N>N^*$ and $s$ close to $1$.\\
{\bf Proof of Theorem \ref{thm:3}.} Let $N>N^*$. We will prove that there exists $\varepsilon_0>0$ such that for any $s\in(1-\varepsilon_0,1)$, the functional $E_{N,s}(u)$  admits no  positive minimizer $u_s$ in the interior of $A_{t_s}$.  Suppose on the contrary that for any $\varepsilon\in (0,1)$, there exists a $s\in(1-\varepsilon,1)$ such that  $E_{N,s}(u)$  admits a positive minimizer $u_s$ in the interior of $A_{t_s}$.
Then $u_s$ satisfies
\begin{equation}\label{eq:3.32}
(-\Delta)^su_s+|x|^2u_s=|u_s|^2u_s+\mu_su_s,
\end{equation}
where $\mu_s\in \mathbb{R}$ is the Lagrange multiplier. 
The Pohozaev identity in Lemma \ref{lem:2.1} yields that
\begin{equation}\label{eq:3.35}
s\int_{\mathbb{R}^2}|(-\Delta)^{\frac s2}u_s|^2\,dx-\int_{\mathbb{R}^2}|x|^2|u_s|^2\,dx-\frac12\int_{\mathbb{R}^2}|u_s|^4\,dx=0,
\end{equation}
which implies that
\begin{equation*}
\int_{\mathbb{R}^2}|x|^2|u_s|^2\,dx\leq s\int_{\mathbb{R}^2}|(-\Delta)^{\frac s2}u_s|^2\,dx\leq st_s.
\end{equation*}
We derive from the H\"{o}lder inequality that
\begin{equation}\label{eq:3.36}
\begin{split}
&\int_{\mathbb{R}^2}|x|^{2s}|u_s|^2\,dx\int_{\mathbb{R}^2}|(-\Delta)^{\frac s2}u_s|^2\,dx \\
&\leq
\Big(\int_{\mathbb{R}^2}|x|^{2}|u_s|^2\,dx\Big)^s\Big(\int_{\mathbb{R}^2}|u_s|^2\,dx\Big)^{1-s}\int_{\mathbb{R}^2}|(-\Delta)^{\frac s2}u_s|^2\,dx \\
&\leq s^st_s^{1+s}N^{1-s}\\
&=s^s(\frac{N}{2s-1})^s\Big(\frac{N_s^*}{N}\Big)^{\frac{s^2}{1-s}}N^{1-s}.
\end{split}
\end{equation}
We can choose $\varepsilon>0$ small enough so that for any $s\in(1-\varepsilon,1)$,
\begin{equation}\label{eq:3.37}
4^{-s}s^s(\frac{N}{2s-1})^s\Big(\frac{N_s^*}{N}\Big)^{\frac{s^2}{1-s}}(1-s)^{-2}N^{1-s}< (1-e^{-1})^2(2e)^{-2},
\end{equation}
which is guaranteed by $N>N^*$ and  $N_s^*\to N^*$ as $s\to1_-$.

On the other hand, by the Hardy inequality, see Lemma 3.2 in \cite{Frank2008}, and the H\"{o}lder inequality, we deduce that
\begin{equation}\label{eq:3.38}
\begin{split}
\int_{\mathbb{R}^2}|x|^{2s}|u_s|^2\,dx\int_{\mathbb{R}^2}|(-\Delta)^{\frac s2}u_s|^2\,dx &\geq C_s\int_{\mathbb{R}^2}|x|^{2s}|u_s|^2\,dx\int_{\mathbb{R}^2}\frac{|u_s|^2}{|x|^{2s}}\,dx\\
&\geq C_s(\int_{\mathbb{R}^2}|u_s|^2\,dx)^2\\
&=C_sN^2,
\end{split}
\end{equation}
where
\[
C_s=4^s\Gamma\Big(\frac{1+s}{2}\Big)^2\Gamma\Big(\frac{1-s}{2}\Big)^{-2}.
\]
Since, for $0<x<1$,
\begin{equation}\label{eq:3.39}
\begin{split}
\Gamma(x)&=\int_0^{+\infty}t^{x-1}e^{-t}\,dt\\
&\leq\int_0^{1}t^{x-1}e^{-t}\,dt+\int_1^{+\infty}e^{-t}\,dt\\
&\leq\frac{1}{x}+e^{-1}\leq\frac{e}{x}
\end{split}
\end{equation}
and
\begin{equation}\label{eq:3.40}
\begin{split}
\Gamma(x)\geq\int_0^{1}e^{-t}\,dt=1-e^{-1},
\end{split}
\end{equation}
we find that
\begin{equation}\label{eq:3.41}
\begin{split}
\int_{\mathbb{R}^2}|x|^{2s}|u_s|^2\,dx\int_{\mathbb{R}^2}|(-\Delta)^{\frac s2}u_s|^2\,dx \geq 4^s(1-e^{-1})^2(2e)^{-2}(1-s)^2.
\end{split}
\end{equation}
Eventually, equations \eqref{eq:3.36} and \eqref{eq:3.41} yield
\begin{equation*}
4^{-s}s^s(\frac{N}{2s-1})^s\Big(\frac{N_s^*}{N}\Big)^{\frac{s^2}{1-s}}(1-s)^{-2}N^{1-s}\geq (1-e^{-1})^2(2e)^{-2},
\end{equation*}
a contradiction to \eqref{eq:3.37}.  The assertion follows.   \qquad $\Box$

\bigskip

\bigskip

\section{Mountain pass solution}

\bigskip

In this section we prove the existence of mountain pass  solutions of \eqref{eq:1.12}.

\begin{Lemma}\label{lem:4.1} Let $0<N<N^*$.
For $s<1$ close to $1$, there exist nonnegative $\phi_1$ and $\phi_2$ in  $S(N)$  such that
\[
c_s(N):=\inf_{g\in \Gamma_s}\max_{t\in[0,1]}E_{N,s}(g(t))>\max\{E_{N,s}(\phi_1)),E_{N,s}(\phi_2))\},
\]
where $\Gamma_s=\{g\in C([0,1],S(N)):g(0)=\phi_1,g(1)=\phi_2\}$.
\end{Lemma}

\begin{proof} The proof lies on finding proper $\phi_1$ and $\phi_2$ such that, each path connecting $\phi_1$ and $\phi_2$ intersects $\partial A_{t_s}$.

Let $\varphi_s(x)=\sqrt{N}\|Q_s\|_2^{-1}Q_s(x)$ and $\varphi_s^t=t\varphi_s(tx)$.  Then, $\varphi_s^t\in S(N)$. By Lemma \ref{lem:2.5}, we have
\begin{equation}\label{eq:4.1}
\int_{\mathbb{R}^2}|(-\Delta)^{\frac s2}\varphi_s^t|^2\,dx=\frac{N}{2s-1}t^{2s}
\end{equation}
and
\begin{equation}\label{eq:4.2}
\begin{split}
E_{N,s}(\varphi_s^t)&=\frac12t^{2s}\int_{\mathbb{R}^2}|(-\Delta)^{\frac s2}\varphi_s|^2\,dx+
\frac12t^{-2}\int_{\mathbb{R}^2}|x|^2|\varphi_s|^2\,dx-\frac14t^2\int_{\mathbb{R}^2}|\varphi_s|^4\,dx\\
&=\frac{N}{2(2s-1)}t^{2s}+\frac{N}{2N_s^*}t^{-2}\int_{\mathbb{R}^2}|x|^2|Q_s|^2\,dx-\frac{sN^2}{2(2s-1)N_s^*}t^{2}.
\end{split}
\end{equation}
Using Lemma \ref{lem:2.3}, we get
\begin{equation}\label{eq:4.3}
\begin{split}
\int_{\mathbb{R}^2}|x|^2|Q_s|^2\,dx&\leq \int_{\{|x|\leq 1\}}|Q_s|^2\,dx+\int_{\{|x|\geq 1\}}|x|^2|Q_s|^2\,dx\\
&\leq N_s^*+\int_{\{|x|\geq 1\}}C_4Q_s\,dx\\
&\leq N_s^*+\int_{\mathbb{R}^2}C_4Q_s\,dx\leq N_s^*+C_3C_4.
\end{split}
\end{equation}
By \eqref{eq:4.2} and \eqref{eq:4.3},
\begin{equation*}
E_{N,s}(\varphi_s^t)\leq \frac{N}{2(2s-1)}t^{2s}+\frac{N}{2N_s^*}t^{-2}(N_s^*+C_3C_4)-\frac{sN^2}{2(2s-1)N_s^*}t^{2}.
\end{equation*}
By Lemma \ref{lem:2.5},  $N^*_s\to N^*$ as $s\to1_-$,  there exists $\eta_k>0$
such that if $s\in (1-\eta_k,1)$, we have
\begin{equation*}
E_{N,s}(\varphi_s^t)\leq Nt^{2s}+\frac{N}{N^*}(N^*+C_3C_4)t^{-2}-\frac{N^2}{4N^*}t^{2}.
\end{equation*}
Choose $\tau_1=1$ and $\tau_2=2^{(1-s)^{-2}}$. Observing that $t_s=\frac{N}{2s-1}\Big(\frac{N_s^*}{N}\Big)^{\frac{s}{1-s}}\to+\infty$ for $0<N<N^*$ and $s\to1_-$, we infer that
 \begin{equation}\label{eq:4.4}
 E_{N,s}(\varphi_s^{\tau_1})\leq N+\frac{N}{N^*}(N^*+C_3C_4)-\frac{N^2}{4N^*}<\frac{N(1-s)}{2s-1}\Big(\frac{N_s^*}{N}\Big)^{\frac{s}{1-s}}
\end{equation}
and
\begin{equation*}\label{eq:4.5}
\int_{\mathbb{R}^2}|(-\Delta)^{\frac s2}\varphi_s^{\tau_1}|^2\,dx=\frac{N}{2s-1}<t_s.
\end{equation*}
Furthermore, since
\begin{equation*}
 \limsup_{s\to1_-}E_{N,s}(\varphi_s^{\tau_2})\to-\infty
\end{equation*}
and
\begin{equation*}
 \lim_{s\to1_-}\frac{2^{\frac{2s}{(1-s)^2}}}{\Big(\frac{N_s^*}{N}\Big)^{\frac{s}{1-s}}}=+\infty,
\end{equation*}
we have, for $s<1$ and $s$ close to $1$,  that $E_{N,s}(\varphi_s^{\tau_2})<0$ and
\begin{equation*}
\int_{\mathbb{R}^2}|(-\Delta)^{\frac s2}\varphi_s^{\tau_2}|^2\,dx=\frac{N}{2s-1}\tau_2^{2s}=\frac{N}{2s-1}2^{\frac{2s}{(1-s)^2}}>\frac{N}{2s-1}\Big(\frac{N_s^*}{N}\Big)^{\frac{s}{1-s}}=t_s.
\end{equation*}
Choose $\phi_1=\varphi_s^{\tau_1}$ and $\phi_2=\varphi_s^{\tau_2}$. For any $g\in \Gamma_s$, \eqref{eq:3.3} yields that
\begin{equation}\label{eq:4.5}
\max_{t\in[0,1]}E_{N,s}(g(t))\geq \inf_{u\in \partial A_{t_s}}E_{N,s}(u)\geq \frac{N(1-s)}{2(2s-1)}\Big(\frac{N_s^*}{N}\Big)^{\frac{s}{1-s}}.
\end{equation}
 The conclusion immediately follows from  \eqref{eq:4.4} and $E_{N,s}(\varphi_s^{\tau_2})<0$.
\end{proof}

\bigskip

\begin{Remark}\label{rem:4.1}
Inequality \eqref{eq:4.5} implies that the mountain pass value $c_s\geq \frac{1-s}{2}t_s$ if $s<1$ and $s$ close to $1$.
\end{Remark}

\bigskip

\begin{Lemma}\label{lem:4.2} Let $0<N<N^*$. For $s<1$ close to $1$, there exist a bounded sequence $\{u_n\}\subset S(N)$, and  a nonnegative sequence $\{v_n\}\subset S(N)$, which is radially symmetric and  decreasing with respect to $|x|$ such that,
\begin{equation}\label{eq:4.6}
E_{N,s}(u_n)\to c_s,
\end{equation}

\begin{equation}\label{eq:4.7}
\|E'_{N,s}|_{S(N)}(u_n)\|_{\mathcal{H}_s^{-1}}\to 0,
\end{equation}

\begin{equation}\label{eq:4.8}
P_{N,s}(u_n):=s\int_{\mathbb{R}^2}|(-\Delta)^{\frac s 2}u_n|^2\,dx-\int_{\mathbb{R}^2}|x|^2|u_n|^2\,dx-\frac12\int_{\mathbb{R}^2}|u_n|^4\,dx\to 0
\end{equation}
as well as
\begin{equation}\label{eq:4.9}
\|u_n-v_n\|_{\mathcal{H}_s}\to0
\end{equation}
as $n\to \infty$, where $\mathcal{H}_s^{-1}$ is the dual space of $\mathcal{H}_s$.
\end{Lemma}

\bigskip

\begin{proof}
Denote by $Y$ the space $\mathcal{H}_s\times\mathbb{R}$ with the norm
\[
\|\cdot\|_Y^2:=\|\cdot\|_{\mathcal{H}_s}^2+\|\cdot\|_{\mathbb{R}}^2.
\]
Let $H(u,t)=e^tu(e^tx)$ and define the functional $\mathcal{E}_{N,s}(u,t):Y\to \mathbb{R}$ by
\begin{equation}\label{eq:4.10}
\begin{split}
\mathcal{E}_s(u,t)&:=E_{N,s}(H(u,t))\\
&=\frac12e^{2st}\int_{\mathbb{R}^2}|(-\Delta)^{\frac s2}u|^2\,dx+\frac 12e^{-2t}\int_{\mathbb{R}^2}|x|^2|u|^2\,dx
-\frac14e^{2t}\int_{\mathbb{R}^2}|u|^4\,dx.
\end{split}
\end{equation}

Using $\phi_1$, $\phi_2$ found in Lemma \ref{lem:4.1}, we define
\[
\mathcal{P}_s=\{\gamma\in C([0,1],S(N)\times \mathbb{R}):\gamma(0)=(\phi_1,0),\gamma(1)=(\phi_2,0)\}
\]
and
\[
b_s=\inf_{\gamma\in \mathcal{P}_s}\max_{t\in[0,1]}\mathcal{E}_s(\gamma(t)).
\]

Since $H(\mathcal{P}_s)\subset \Gamma_s$, and for any $g\in \Gamma_s$, we have $(g,0)\in \mathcal{P}_s$  and
$E_{N,s}(g)=\mathcal{E}_s(g,0)$. Therefore, $b_s=c_s$. By Lemma \ref{lem:4.1}, we have
\[
b_s>\max\{\mathcal{E}_s(\gamma(0)),\mathcal{E}_s(\gamma(1))\}.
\]
Moreover,  Lemma \ref{lem:4.1} allows us to find  a sequence of paths $\{h_n\}\subset \Gamma_s$ such that
\[
\lim_{n\to\infty}\max_{t\in[0,1]}E_{N,s}(h_n(t))=c_s=b_s.
\]
Obviously, $|h_n|\in \Gamma_s$ and
$$
c_s\leq \max_{t\in[0,1]}E_{N,s}(|h_n|(t))\leq \max_{t\in[0,1]}E_{N,s}(h_n(t)).
$$
Let $|h_n|^*$ be  the Schwartz rearrangement of $|h_n|$ with respect to $x$. Apparently, $|h_n|^*\in \Gamma_s$. Arguing as  \eqref{eq:3.12}, we have
\[
c_s\leq\max_{t\in[0,1]}E_{N,s}(|h_n|^*(t))\leq \max_{t\in[0,1]}E_{N,s}(|h_n|(t)).
\]
Hence
\begin{equation}\label{eq:4.11}
\lim_{n\to\infty}\max_{t\in[0,1]}E_{N,s}(|h_n|^*(t))=c_s=b_s
\end{equation}
and
\[
\lim_{n\to\infty}\max_{t\in[0,1]}\mathcal{E}_{s}(|h_n|^*(t),0)=\lim_{n\to\infty}\max_{t\in[0,1]}E_{N,s}(|h_n|^*(t))=c_s=b_s.
\]
By Theorem 3.2 in \cite{Ghoussoub1993} or Lemma 2.3 in \cite{Jeanjean1997}, there exists $\{(w_n,t_n)\}\subset S(N)\times \mathbb{R}$ such that
\begin{equation}\label{eq:4.12}
\lim_{n\to \infty}\mathcal{E}_{s}(w_n,t_n)=b_s;
\end{equation}
\begin{equation}\label{eq:4.13}
\|\mathcal{E}_{s}'|_{S(N)\times \mathbb{R}}(w_n,t_n)\|_{Y^{-1}}\to 0
\end{equation}
and
\begin{equation}\label{eq:4.14}
\lim_{n\to \infty}dist((w_n,t_n),(|h_n|^*(t),0))=0.
\end{equation}
Let $u_n=H(w_n,t_n)$. We see from  \eqref{eq:4.12} and $b_s=c_s$ that
\begin{equation}\label{eq:4.14a}
E_{N,s}(u_n)\to c_s
\end{equation}
as $n\to \infty$. Equation \eqref{eq:4.13} yields that
\begin{equation}\label{eq:4.15}
\langle\mathcal{E}_{s}'(w_n,t_n),z\rangle_{Y^{-1},Y}=\langle\mathcal{E}_{s}'|_{S(N)\times \mathbb{R}}(w_n,t_n),z\rangle_{Y^{-1},Y}=o(\|z\|_{Y}),
\end{equation}
for all $z\in T_{(w_n,t_n)}(S(N)\times \mathbb{R}):=\{(z_1,z_2)\in Y:(w_n,z_1)_2=0\}$, where $(\cdot,\cdot)_2$ is the inner product of $L^2(\mathbb{R}^2)$.

Choosing $z=(0,1)$ in \eqref{eq:4.15}, we deduce that
\begin{equation}\label{eq:4.16}
\begin{split}
\quad &\langle\mathcal{E}_{s}'(w_n,t_n),(0,1)\rangle_{Y^{-1},Y}\\
&=\frac d{dt}\mathcal{E}_{s}(w_n,t_n+t)|_{t=0}\\
&=se^{2st_n}\int_{\mathbb{R}^2}|(-\Delta)^{\frac s2}w_n|^2\,dx-e^{2t_n}\int_{\mathbb{R}^2} |x|^2|w_n|^2\,dx
-\frac12e^{2t_n}\int_{\mathbb{R}^2} |w_n|^4\,dx\\
&=s\int_{\mathbb{R}^2}|(-\Delta)^{\frac s2}u_n|^2\,dx-\int_{\mathbb{R}^2} |x|^2|u_n|^2\,dx
-\frac12\int_{\mathbb{R}^2} |u_n|^4\,dx\\
&=P_{N,s}(u_n)\to0
\end{split}
\end{equation}
as $n\to \infty$.

Now, by \eqref{eq:4.14}, $t_n\to 0$ as $n\to\infty$, and there exists $\tau_n\in[0,1]$  such that
\[
\lim_{n\to \infty}\|w_n-|h_n|^*(\tau_n)\|_{\mathcal{H}_s}=0
\]
Hence,
$$
\|u_n-v_n\|_{\mathcal{H}_s}^2\to 0
$$
 as $n\to \infty$, where
$v_n=H(|h_n|^*(\tau_n),t_n)\geq0$ and $v_n\in H_r^{s}(\mathbb{R}^2)$.

For any $\varphi\in T_{u_n}(S(N)):=\{u\in \mathcal{H}_s, (u,\varphi)_2=0\}$, the function $\psi_n=e^{-t_n}\varphi(e^{-t_n}x)$ belongs to
$T_{(w_n,t_n)(S(N)\times \mathbb{R})}$. Indeed,
\begin{equation*}
(w_n,\psi_n)_2=\int_{\mathbb{R}^2}w_n(x)\psi_n(x)\,dx=\int_{\mathbb{R}^2}e^{t_n}w_n(e^{t_n}x)\varphi(x)\,dx=(u_n,\varphi)_2=0.
\end{equation*}
Therefore,  for any $\varphi\in T_{u_n}(S(N))$, we derive from \eqref{eq:4.12} and \eqref{eq:4.14} that,
\begin{equation}\label{eq:4.17}
\begin{split}
\langle E'_{N,s}|_{S(N)}(u_n),\varphi\rangle_{\mathcal{H}_s^{-1},\mathcal{H}_s}&=\langle E'_{N,s}(u_n),\varphi\rangle_{\mathcal{H}_s^{-1},\mathcal{H}_s}\\
&=\frac{d}{dt}E_{N,s}(u_n+t\varphi)|_{t=0}\\
&=\frac{d}{dt}E_{N,s}(e^{t_n}w_n(e^{t_n}x)+te^{t_n}\psi_n(e^{t_n}x))\Big|_{t=0}\\
&=\frac{d}{dt}\mathcal{E}_{s}(w_n+t\psi_n,t_n)|_{t=0}\\
&=\langle\mathcal{E}'_{s}(w_n,t_n),(\psi_n,0)\rangle=o(\|\psi_n\|_{\mathcal{H}_s}^2).
\end{split}
\end{equation}
Since $t_n\to 0$ as $n\to\infty$,  for any $\varphi\in T_{u_n}(S(N))$ we have
\begin{equation}\label{eq:4.18}
\langle E'_{N,s}|_{S(N)}(u_n),\varphi\rangle_{\mathcal{H}_s^{-1},\mathcal{H}_s}=\langle E'_{N,s}(u_n),\varphi\rangle_{\mathcal{H}_s^{-1},\mathcal{H}_s}=o(\|\varphi\|_{\mathcal{H}_s}^2),
\end{equation}
that is, $\|E'_{N,s}|_{S(N)}(u_n)\|_{\mathcal{H}_s^{-1}}\to0$, as $n\to\infty$.

Finally we prove the boundedness of $\{u_n\}$ in $\mathcal{H}_s$. By \eqref{eq:4.14a} and \eqref{eq:4.16},
\begin{equation}\label{eq:3.14a}
\begin{split}
(1-s)\int_{\mathbb{R}^2}|(-\Delta)^{\frac s2}u_n|^2\,dx+2\int_{\mathbb{R}^2}|x|^2|u_n|^2\,dx=
2E_{N,s}(u_n)-F_{N,s}(u_n)\to c_s,\quad as \quad n\to\infty.
\end{split}
\end{equation}
This immediately yields  the boundedness of $\{u_n\}$ in $\mathcal{H}_s$.

The proof is complete.

\end{proof}

\bigskip

{\bf Proof of Theorem \ref{thm:4}.} By Lemma \ref{lem:4.2}, there exists a bounded sequence $\{u_n\}\in S(N)$ verifying \eqref{eq:4.6}--\eqref{eq:4.9}.
So we may assume that $u_n\rightharpoonup v_s$ weakly in $\mathcal{H}_s$. It follows from Lemma \ref{lem:2.7} that $u_n\to v_s$ strongly in $L^2(\mathbb{R}^2)\cap L^4(\mathbb{R}^2)$.

For any $\psi\in \mathcal{H}_s$, we have  $\varphi_n=\psi-N^{-1}\langle u_n,\psi\rangle u_n\in T_{u_n}(S(N))$. Choosing $\varphi=\varphi_n$ in \eqref{eq:4.18}, we have
\begin{equation}\label{eq:4.19}
(-\Delta)^su_n+|x|^2u_n-|u_n|^2u_n-N^{-1}\langle E'_{N,s}(u_n),u_n\rangle u_n\to 0\quad {\rm in} \quad
\mathcal{H}_s^{-1}.
\end{equation}
Since $\langle E'_{N,s}(u_n),u_n\rangle$ is uniformly bounded in  $n$, we may assume
\[
N^{-1}\langle E'_{N,s}(u_n),u_n\rangle \to \mu_s
\]
as $n\to \infty$. Thus,
\begin{equation}\label{eq:4.20}
(-\Delta)^sv_s+|x|^2v_s-|v_s|^2v_s-\mu_sv_s=0\quad {\rm in} \quad
\mathcal{H}_s^{-1}
\end{equation}
as well as  $\mu_s=N^{-1}\langle E'_{N,s}(v_s),v_s\rangle,$ alternatively,
\begin{equation}\label{eq:4.21}
\langle E'_{N,s}(u_n),u_n\rangle \to \langle E'_{N,s}(v_s),v_s\rangle
\end{equation}
as $n\to \infty$. Using the fact that $u_n\to v_s$ strongly in $L^2(\mathbb{R}^2)\cap L^4(\mathbb{R}^2)$ and \eqref{eq:4.21}, we derive $u_n\to v_s$ strongly in $\mathcal{H}_s$. Therefore, $E_{N,s}(v_s)=c_s$. Furthermore, by \eqref{eq:4.9} we conclude that $v_n\to v_s$ strongly in $\mathcal{H}_s$, where $v_n\geq0$  is  radially symmetric and decreasing with respect to $|x|$. Hence, $v_s$ is also radially symmetric and decreasing with respect to $|x|$. Moreover, $v_s$ is nonnegative and nontrivial.
The proof is complete. \quad $\Box$

\bigskip

\bigskip

\section{Asymptotic behavior}

\bigskip

In this section, we study the asymptotic behavior of  the local minimal solution $u_s$ and the mountain pass solution $v_s$ as $s\to 1_-$.

Since both the fractional operator $(-\Delta)^s$ and the solutions  $u_s$  and $v_s$ of \eqref{eq:1.12} depend on $s$, it becomes  difficult to deal with asymptotic behavior of $u_s$  and $v_s$ as $s\to 1_-$.  In the consideration of passing limit, a key ingredient is to find the limit equation of \eqref{eq:1.12} as $s\to 1_-$. Taking into account the result in Lemma \ref{lem:3.3}, we introduce  very weak solutions in the  following sense.
\begin{Definition} \label{very weak}
We say $u\in L^2(\mathbb{R}^2)\cap L^4(\mathbb{R}^2)$ is a very weak solution of \eqref{eq:1.11}, if $xu\in L^2(\mathbb{R}^2)$ and
\[
\int_{\mathbb{R}^2}u(-\Delta)^{s}\varphi\,dx+\int_{\mathbb{R}^2}|x|^2u\varphi\,dx
-\int_{\mathbb{R}^2}|u|^2u\varphi\,dx-\mu_{{}_N}\int_{\mathbb{R}^2}u\varphi\,dx=0
\]
holds for any $\varphi\in C_c^\infty(\mathbb{R}^2)$.
\end{Definition}
This definition enlightens a way to pass limit in $s$.

\bigskip

We commence with the study of $u_s$.

{\bf Proof of Theorem \ref{thm:5}.}
 By \eqref{eq:1.23}, arguing as \eqref{eq:3.24}, we derive that for $s>3/4$,
 $$
 \int_{\mathbb{R}^2}|(-\Delta)^{\frac 38}u_s|^2\,dx\leq C,
 $$
 where $C$ is independent of $s$. This with \eqref{eq:1.24} implies that $\{u_s\}$ is uniformly bounded in $\mathcal{H}_{3/4}$. By  Lemma \ref{lem:2.7},   there exist a sequence $\{s_k\}$,  $s_k\to1_-$ as $k\to\infty$,  and $u\in\mathcal{H}_{3/4}$ such that $u_{s_k}\rightharpoonup u$ weakly in $\mathcal{H}_{3/4}$,  $u_{s_k}\to u$ strongly in $L^2(\mathbb{R}^2)\cap L^4(\mathbb{R}^2)$. By \eqref{eq:3.30}, $\mu_k$ is bounded in $k$. We may assume $\mu_k\to \mu$ as $k\to\infty$. Therefore, $u$ solves
\begin{equation}\label{eq:5.1}
-\Delta u+|x|^2u=|u|^2u+\mu u
\end{equation}
in very weak sense.

Obviously, $\hat{u}_{s_k}\to \hat{u}$ strongly in $L^2(\mathbb{R}^2)$ and so $\hat{u}_{s_k}\to \hat{u}$ a.e. $\mathbb{R}^2$.  By the Fatou lemma and Lemma 3.1 in \cite{Frank2008}, we have
\begin{equation}\label{eq:5.2}
\begin{split}
\int_{\mathbb{R}^2}|\xi|^2|\hat{u}(\xi)|^2\,d\xi\leq \liminf_{k\to\infty}\int_{\mathbb{R}^2}|\xi|^{2s_k}|\hat{u}_{s_k}(\xi)|^2\,d\xi
=\liminf_{k\to\infty}\int_{\mathbb{R}^2}\Big(|(-\Delta)^{\frac{s_k}{2}}u_{s_k}|^2\,dx.
\end{split}
\end{equation}
Therefore, $u\in \mathcal{H}$ and solves \eqref{eq:5.1} in weak sense.

Since $u_{s_k}$ is a solution of \eqref{eq:3.27},  we show by \eqref{eq:5.1} and the fact  that $u_{s_k}\to u$ strongly in $L^2(\mathbb{R}^2)\cap L^4(\mathbb{R}^2)$, that
 \begin{equation}\label{eq:5.3}
\int_{\mathbb{R}^2}\Big(|(-\Delta)^{\frac{s_k}{2}}u_{s_k}|^2+|x|^2|u_{s_k}|^2\Big)\,dx\to
\int_{\mathbb{R}^2}\Big(|\nabla u|^2+|x|^2|u|^2\Big)\,dx
\end{equation}
as $k\to\infty$.

Finally, we prove that $u$ is a global minimizer of $E(u)$ on $S(N)$. If not,
\begin{equation}\label{eq:5.4}
E(u)>e_1=\inf_{u\in \mathcal{H}}E(u).
\end{equation}
It is known in \cite{Guo2014} that there exists $w\in \mathcal{H}$ such that $E(w)=e_1$. Noting $t_s\to+\infty$, as $s\to1_-$, we have
$w\in A_{t_s}$ for $s$ close to $1$. It follows from \eqref{eq:5.3}, \eqref{eq:5.4} and
\[
\int_{\mathbb{R}^2}|(-\Delta)^{\frac{s}{2}}w|^2\,dx\to \int_{\mathbb{R}^2}|\nabla w|^2\,dx,\quad {\rm as} \quad s\to1_-
\]
 that
 \[
 E_{N,s_k}(u_{s_k})>E_{N,s_k}(w)
 \]
for $k$ large enough, which contradicts to the assumption that $u_{s_k}$ is a minimizer of $E_{N,s_k}(u)$ in $A_{t_{s_k}}$. The proof is complete.
\quad $\Box$

\bigskip

Next, we turn to consider the asymptotic behavior of  $v_s$  as $s\to1_-$. We start by estimating the mountain pass value $c_s$.
\begin{Lemma}\label{lem:5.1} Let $\tilde{c}_s=\frac{N(1-s)}{2(2s-1)}\Big(\frac{N_s^*}{N}\Big)^{\frac{s}{1-s}}$. There holds that
\begin{equation*}
0\leq \limsup_{s\to1_-}[c_s-\tilde{c}_s](N_s^*/N)^{\frac1{1-s}}\leq C,
\end{equation*}
where $C>0$ is independent of $s$.
\end{Lemma}
\begin{proof}
By Remark \ref{rem:4.1}, the first inequality holds. We need only to prove the second inequality.
Let $g(t)=\varphi_s^{(1-t)\tau_1+t\tau_2},t\in[0,1]$, where
$\varphi_s$, $\tau_1$ and $\tau_2$ are given in the proof of Lemma \ref{lem:4.1}.
Obviously, $g(0)=\varphi_s^{\tau_1}=\phi_1$, $g(1)=\varphi_s^{\tau_2}=\phi_2$ and $g\in \Gamma_s$. Hence,
\begin{equation}\label{eq:5.5}
c_s\leq \max_{t\in[0,1]}E_{N,s}(g(t)).
\end{equation}
Let $\rho= (1-t)\tau_1+t\tau_2$. By \eqref{eq:4.2}, we have
\begin{equation}\label{eq:5.6}
\begin{split}
E_{N,s}(g(t))&=\frac{N}{2(2s-1)N_s^*}[N_s^*\rho^{2s}+(2s-1)\rho^{-2}\int_{\mathbb{R}^2}|x|^2|Q_s|^2\,dx
-sN\rho^{2}]\\
&=:\frac{N}{2(2s-1)N_s^*}f(\rho).
\end{split}
\end{equation}
By Lemma \ref{lem:4.1} and Remark \ref{rem:4.1},
\[
\max_{t\in[0,1]}E_{N,s}(g(t))\geq \frac{N(1-s)}{2(2s-1)}(\frac{N_s^*}N)^{\frac s{1-s}}>\max\{E_{N,s}(g(0)),E_{N,s}(g(1))\}.
\]
This yields
\begin{equation}\label{eq:5.7}
\max_{\rho\in [\tau_1,\tau_2]}f(\rho)\geq N_s^*(1-s)(\frac{N_s^*}N)^{\frac s{1-s}}>\max\{f(\tau_1),f(\tau_2)\}.
\end{equation}
Therefore, the function $f(\rho)$ achieves its maximum in $[\tau_1,\tau_2]$ at a point  $\rho_0\in (\tau_1,\tau_2)$. It means  that
\begin{equation}\label{eq:5.8}
f'(\rho_0)=2sN_s^*\rho_0^{2s-1}-2(2s-1)\rho_0^{-3}\int_{\mathbb{R}^2}|x|^2|Q_s|^2\,dx-2sN\rho_0=0
\end{equation}
and
\begin{equation}\label{eq:5.9}
f(\rho_0)=\max_{\rho\in[\tau_1,\tau_2]}f(\rho)
=N_s^*{\rho_0}^{2s}+(2s-1){\rho_0}^{-2}\int_{\mathbb{R}^2}|x|^2|Q_s|^2\,dx
-sN\rho_0^2.
\end{equation}
As a result,
\begin{equation}\label{eq:5.9a}
\begin{split}
f(\rho_0)&=(1-s)N_s^*{\rho_0}^{2s}+2(2s-1){\rho_0}^{-2}\int_{\mathbb{R}^2}|x|^2|Q_s|^2\,dx\geq
N_s^*(1-s)\bigg(\frac{N_s^*}N\bigg)^{\frac s{1-s}}.
\end{split}
\end{equation}
By \eqref{eq:4.3} and Lemma \ref{lem:2.5},  $\int_{\mathbb{R}^2}|x|^2|Q_s|^2\,dx\leq C$.  This inequality with  $\rho_0\geq1$ and \eqref{eq:5.9a} yield
\begin{equation}\label{eq:5.10}
\liminf_{s\to 1_-}\frac{{\rho_0}^2}{(\frac{N_s^*}N)^{\frac s{1-s}}}\geq \liminf_{s\to 1_-}\frac{{\rho_0}^{2s}}{(\frac{N_s^*}N)^{\frac s{1-s}}}\geq1.
\end{equation}

Let $h(\rho)=N_s^*\rho^{2s}
-sN\rho^{2}$ and $\rho_1={(\frac{N_s^*}N)^{\frac 1{2(1-s)}}}$. Then,   $h(\rho)$ is increasing in $[1,\rho_1]$ and decreasing in $[\rho_1,\tau_2]$ whenever $s<1$ close to $1$. Moreover,
\begin{equation}\label{eq:5.11}
\max_{\rho\in[\tau_1,\tau_2]}h(\rho)=h(\rho_1)=N_s^*(1-s)\bigg(\frac{N_s^*}N\bigg)^{\frac s{1-s}}.
\end{equation}
By \eqref{eq:5.5} , \eqref{eq:5.6} and  and  \eqref{eq:5.11}, we have
\begin{equation*}
\begin{split}
c_s&\leq \frac{N}{2(2s-1)N_s^*}\max_{\rho\in [\tau_1,\tau_2]}f(\rho)\\
&=\frac{N}{2(2s-1)N_s^*}f(\rho_0)\\
&\leq \frac{N}{2(2s-1)N_s^*}\max_{\rho\in [\tau_1,\tau_2]}h(\rho)+\frac{N}{2N_s^*}\rho_0^{-2}\int_{\mathbb{R}^2}|x|^2|Q_s|^2\,dx\\
&=\tilde{c}_s+\frac{N}{2N_s^*}\rho_0^{-2}\int_{\mathbb{R}^2}|x|^2|Q_s|^2\,dx.
\end{split}
\end{equation*}
This with \eqref{eq:4.3} implies that
\[
[c_s-\tilde{c}_s]\rho_0^{2}\leq \frac{N}{2N_s^*}\int_{\mathbb{R}^2}|x|^2|Q_s|^2\,dx\leq
\frac{N}{2N_s^*}(N_s^*+C_3C_4).
\]
Equation  \eqref{eq:5.10} and Lemma \ref{lem:2.5} then yield
\[
\limsup_{n\to\infty}[c_s-\tilde{c}_s](N_s^*/N)^{\frac1{1-s}}\leq C,
\]
where $C>0$ is independent of $s$.
The assertion follows.

\end{proof}

\bigskip

Let $\lambda_1$ be  the principle eigenvalue of the operator $\mathcal{L}:=(-\Delta)^s+|x|^2$ and  $\varphi>0$ be the  corresponding eigenfunction of $\mathcal{L}$ with
$\|\varphi\|_2^2=1$.

\begin{Lemma}\label{lem:5.2} There hold that
\begin{equation}\label{eq:5.12}
\frac{N(1-s)}{2s-1}\big(\frac{N_s^*}N\big)^{\frac s{1-s}}\leq \int_{\mathbb{R}^2}|(-\Delta)^{\frac s2}v_s|^2\,dx\leq 2N\big(\frac{N_s^*}N\big)^{\frac s{1-s}}
\end{equation}
and
\begin{equation}\label{eq:5.13}
 1-\frac{\lambda_1N}{\int_{\mathbb{R}^2}|(-\Delta)^{\frac{s}{2}}v_s|^2\,dx}\leq\frac{\int_{\mathbb{R}^2}|v_s|^4\,dx}{\int_{\mathbb{R}^2}|(-\Delta)^{\frac{s}{2}}v_s|^2\,dx}\leq 2s.
\end{equation}
provided that $s<1$ and close to $1$.
\end{Lemma}

\begin{proof}
Since $v_s$ satisfies \eqref{eq:1.12}  with $\mu_s=N^{-1}\langle E'_{N,s}(v_s),v_s \rangle$ and the Pohozaev identity in Lemma \ref{lem:2.1}, we have
\begin{equation}\label{eq:5.14}
s\int_{\mathbb{R}^2}|(-\Delta)^{\frac{s}{2}}v_s|^2\,dx-\int_{\mathbb{R}^2}|x|^2|v_s|^2\,dx
-\frac12\int_{\mathbb{R}^2}|v_s|^4\,dx=0.
\end{equation}
This and  Lemma \ref{lem:5.1} imply that
\begin{equation}\label{eq:5.15}
\begin{split}
\frac{N(1-s)}{2(2s-1)}\big(\frac{N_s^*}N\big)^{\frac s{1-s}}&\leq E_{N,s}(v_s)\\
&\leq\frac{1-s}2\int_{\mathbb{R}^2}|(-\Delta)^{\frac{s}{2}}v_s|^2\,dx+\int_{\mathbb{R}^2}|x|^2|v_s|^2\,dx\\
&\leq \frac{N(1-s)}{2(2s-1)}\big(\frac{N_s^*}N\big)^{\frac s{1-s}}+O(\big(\frac{N_s^*}N\big)^{\frac s{s-1}})
\end{split}
\end{equation}
for $s<1$ and close to $1$. Hence,
\begin{equation}\label{eq:5.16}
\int_{\mathbb{R}^2}|(-\Delta)^{\frac{s}{2}}v_s|^2\,dx\leq2 N\big(\frac{N_s^*}N\big)^{\frac s{1-s}}
\end{equation}
for $s<1$ and close to $1$.

On the other hand, we know from \eqref{eq:5.14} and Remark \ref{rem:4.1} that
\begin{equation}\label{eq:5.17}
\int_{\mathbb{R}^2}|(-\Delta)^{\frac{s}{2}}v_s|^2\,dx\geq 2E_{N,s}(v_s)\geq \frac{N(1-s)}{2s-1}\big(\frac{N_s^*}N\big)^{\frac s{1-s}}.
\end{equation}
Thus, we conclude that \eqref{eq:5.12} is valid.

\bigskip

Now, we prove \eqref{eq:5.13}.

Since $v_s$ is a solution of \eqref{eq:1.12}, we infer that
\begin{equation*}
\begin{split}
\lambda_1\int_{\mathbb{R}^2}v_s\varphi\,dx&=\int_{\mathbb{R}^2}v_s[(-\Delta)^s\varphi+|x|^2\varphi]\,dx\\
&=\int_{\mathbb{R}^2}\varphi[(-\Delta)^sv_s+|x|^2v_s]\,dx\geq \mu_s\int_{\mathbb{R}^2}v_s\varphi\,dx,
\end{split}
\end{equation*}
that is,
\begin{equation}\label{eq:5.18}
\mu_s\leq \lambda_1
\end{equation}
Hence, we deduce from \eqref{eq:4.20} that
\begin{equation*}
\begin{split}
\int_{\mathbb{R}^2}|v_s|^4\,dx&=\int_{\mathbb{R}^2}|(-\Delta)^{\frac{s}{2}}v_s|^2\,dx
+\int_{\mathbb{R}^2}|x|^2|v_s|^2\,dx-\mu_sN\\
&\geq \int_{\mathbb{R}^2}|(-\Delta)^{\frac{s}{2}}v_s|^2\,dx+\int_{\mathbb{R}^2}|x|^2|v_s|^2\,dx-\lambda_1N\\
&\geq \int_{\mathbb{R}^2}|(-\Delta)^{\frac{s}{2}}v_s|^2\,dx-\lambda_1N.
\end{split}
\end{equation*}
This  yields that
\begin{equation}\label{eq:5.19}
\frac{\int_{\mathbb{R}^2}|v_s|^4\,dx}{\int_{\mathbb{R}^2}|(-\Delta)^{\frac{s}{2}}v_s|^2\,dx}\geq 1-
\frac{\lambda_1N}{\int_{\mathbb{R}^2}|(-\Delta)^{\frac{s}{2}}v_s|^2\,dx}.
\end{equation}
While \eqref{eq:5.14} implies that
\begin{equation}\label{eq:5.20}
\frac{\int_{\mathbb{R}^2}|v_s|^4\,dx}{\int_{\mathbb{R}^2}|(-\Delta)^{\frac{s}{2}}v_s|^2\,dx}\leq 2s,
\end{equation}
giving the conclusion. The proof is complete.
\end{proof}

\bigskip

\begin{Remark}\label{rem:5.1}
Equation \eqref{eq:5.12} also tell us that
$$
\lim_{s\to1_-}\int_{\mathbb{R}^2}|(-\Delta)^{\frac{s}{2}}v_s|^2\,dx=+\infty.
$$
\end{Remark}

\bigskip

{\bf Proof of Theorem 1.6.}
Set $\tilde{v}_s=\varepsilon_sv_s(\varepsilon_sx)$ with $\varepsilon_s=\|(-\Delta)^{\frac s2}v_s\|_2^{-\frac1s}$.
Then $\tilde{v}_s$ obeys
\[
\int_{\mathbb{R}^2}|(-\Delta)^{\frac{s}{2}}\tilde{v}_s|^2\,dx=1
\]
and
\begin{equation}\label{eq:5.21}
\int_{\mathbb{R}^2}|\tilde{v}_s|^2\,dx=N.
\end{equation}
As the proof of  \eqref{eq:3.24}, we get
\[
\int_{\mathbb{R}^2}|(-\Delta)^{\frac{3}{8}}\tilde{v}_s|^2\,dx\leq1+N.
\]
Therefore, there exists a sequence $\{s_k\}$,  $s_k\to1_-$ as $k\to\infty$,  such that $\tilde{v}_{s_k}\rightharpoonup \tilde{v}$ weakly in $H^{\frac34}(\mathbb{R}^2)$; $\tilde{v}_{s_k}\to \tilde{v}$ strongly in $L_{loc}^p(\mathbb{R}^2)$
for  $2\leq p<8$. Furthermore, due to the fact that the embedding $H_r^{\frac34}(\mathbb{R}^2)\hookrightarrow L^p(\mathbb{R}^2)$ is compact, see Theorem $\Pi$.1 in \cite{Lions1982}, $\tilde{v}_{s_k}\to \tilde{v}$ strongly in $L^q(\mathbb{R}^2)$ for any $2<p<8$.

Since   $v_s$ is a weak solution of  \eqref{eq:1.12}  with $\mu_s=N^{-1}\langle E'_{N,s}(v_s),v_s \rangle$, so it is a very weak solution, that is,
\begin{equation}\label{eq:5.22}
\int_{\mathbb{R}^2}v_s(-\Delta)^s\varphi\,dx+\int_{\mathbb{R}^2}|x|^2v_s\varphi\,dx
-\int_{\mathbb{R}^2}|v_s|^2v_s\varphi\,dx-\mu_s\int_{\mathbb{R}^2}v_s\varphi\,dx=0.
\end{equation}
Thus, $\tilde{v}_s$ obeys that
\begin{equation}\label{eq:5.23}
\int_{\mathbb{R}^2}\tilde{v}_s(-\Delta)^s\varphi\,dx+\varepsilon_s^{2s}\int_{\mathbb{R}^2}|x|^2\tilde{v}_s\varphi\,dx
-\varepsilon_s^{2(s-1)}\int_{\mathbb{R}^2}|\tilde{v}_s|^2\tilde{v}_s\varphi\,dx
-\mu_s\varepsilon_s^{2s}\int_{\mathbb{R}^2}\tilde{v}_s\varphi\,dx=0.
\end{equation}

Now we study the asymptotic behavior of $\tilde{v}_{s_k}$. For simplicity, we denote
$v_{s_k}$, $\mu_{s_k}$, $\varepsilon_{s_k}$ by $v_{k}$,  $\mu_k$, $\varepsilon_{k}$ respectively.
By Lemma \ref{lem:3.3},  we have
 \begin{equation}\label{eq:5.24}
 \lim_{k\to \infty}\int_{\mathbb{R}^2}\tilde{v}_{s_k}(-\Delta)^{s_k}\varphi\,dx
 =\int_{\mathbb{R}^2}\tilde{v}(-\Delta\varphi)\,dx.
\end{equation}
It follows from Lemma \ref{lem:5.2} that
\begin{equation*}
\begin{split}
\varepsilon_{k}^{2({s_k}-1)}&=[\int_{\mathbb{R}^2}|(-\Delta)^{\frac{s_k}{2}}v_{k}|^2\,dx]^{-\frac{s_k-1}{s_k}}\\
&\geq (2N)^{-\frac{s_k-1}{s_k}}\frac{N}{N^*_{s_k}}\to \frac N{N^*}, \quad as \quad k\to \infty,
\end{split}
\end{equation*}
and
\begin{equation*}
\varepsilon_{k}^{2({s_k}-1)}\leq \Big[\frac{2(2s_k-1)}{N(1-s_k)}\Big]^{\frac{s_k-1}{s_k}}\frac{N}{N^*_{s_k}}\to \frac N {N^*}\quad as  \quad k\to \infty.
\end{equation*}
Hence,
\begin{equation*}
\lim_{k\to \infty}\varepsilon_{k}^{2({s_k}-1)}=\frac N{N^*}.
\end{equation*}
The convergence $\tilde{v}_{s_k}\to \tilde{v}$ strongly in $L^p(\mathbb{R}^2)$, $2\leq p\leq 4$, yields that
\begin{equation}\label{eq:5.25}
 \lim_{k\to \infty}\varepsilon_{k}^{2{s_k}}\int_{\mathbb{R}^2}|x|^2\tilde{v}_{s_k}\varphi\,dx
 =0
\end{equation}
and
\begin{equation}\label{eq:5.26}
\lim_{k\to \infty}\varepsilon_{k}^{2({s_k}-1)}\int_{\mathbb{R}^2}\tilde{v}_{s_k}^3\varphi\,dx
=\frac N{N^*}\int_{\mathbb{R}^2}\tilde{v}^3\varphi\,dx.
\end{equation}
By \eqref{eq:5.18} and $\lim_{s\to1_-}\varepsilon_s=0$ ,
\[
\limsup_{s\to1_-}\mu_s\varepsilon_s^{2s}\leq \lim_{s\to1_-}\lambda_1\varepsilon_s^{2s}=0.
\]
Using $\mu_s=N^{-1}\langle E'_{N,s}(v_s),v_s\rangle$, we show by Lemma \ref{lem:5.2} that
\begin{equation}\label{eq:5.27}
\mu_{k}\varepsilon_{k}^{2s_k}=N^{-1}[
\int_{\mathbb{R}^2}|(-\Delta)^{\frac{s_k}{2}}v_{k}|^2\,dx+
\int_{\mathbb{R}^2}|x|^2|v_{k}|^2\,dx-\int_{\mathbb{R}^2}|v_{k}|^4\,dx]
(\int_{\mathbb{R}^2}|(-\Delta)^{\frac{s_k}{2}}v_{k}|^2\,dx)^{-1}
\end{equation}
and
\begin{equation}\label{eq:5.27a}
\begin{split}
\mu_{k}\varepsilon_{k}^{2s_k}
&\geq -N^{-1}\Big(\int_{\mathbb{R}^2}|(-\Delta)^{\frac{s_k}{2}}v_{k}|^2\,dx\Big)^{s_k-1}
\frac{\int_{\mathbb{R}^2}|v_{k}|^4\,dx}{\int_{\mathbb{R}^2}|(-\Delta)^{\frac{s_k}{2}}v_{k}|^2\,dx}\\
&\geq -2sN^{-1}\Big[\frac{2(2s_k-1)}{N(1-s_k)}\Big]^{\frac{s_k-1}{s_k}}\frac N{N_{s_k}^*}\to -\frac 2{N^*}
\quad {\rm as}  \quad k\to \infty.
\end{split}
\end{equation}
Therefore, we may assume
\begin{equation}\label{eq:5.28}
\lim_{k\to\infty}\mu_{k}\varepsilon_{k}^{2s_k}=-\beta^2\leq0.
\end{equation}
The convergence  $\tilde{v}_{s_k}\rightharpoonup \tilde{v}$ weakly in $L^2(\mathbb{R}^2)$ implies
\begin{equation}\label{eq:5.29}
\lim_{k\to\infty}\mu_{k}\varepsilon_{k}^{2{s_k}}\int_{\mathbb{R}^2}\tilde{v}_{s_k}\varphi\,dx=
-\beta^2\int_{\mathbb{R}^2}\tilde{v}\varphi\,dx.
\end{equation}
Consequently, $\tilde{v}\geq 0$,  $\tilde{v}\neq0$, solves
\[
-\Delta u+\beta^2 u=\frac {N^*}Nu^3
\]
in very weak sense.  We have that $\beta\neq0$ by  the Liouville theorem for the equation
\[
-\Delta u=\frac {N^*}Nu^3
\]
in \cite{GidasSpruck1981}.  By the uniqueness result for \eqref{eq:1.14},
\[
\tilde{v}=\sqrt{\frac{N}{N^*}}\beta Q(\beta x),
\]
and then, $\|\tilde{v}\|_2^2=N$. It implies that $\tilde{v}_{s_k}\to \tilde{v}$ strongly in $L^2(\mathbb{R}^2)$.

\bigskip

Finally, we determine $\beta$.

Let $\psi\in C_c^\infty(\mathbb{R}^2)$ be a function such that $\psi\equiv1$ in $B_1$, $\psi\equiv0$ in $\mathbb{R}^2\setminus B_2$, and $0\leq \psi\leq 1$. By \eqref{eq:5.22}, the function $\psi_R=\psi(R^{-1}x)$  verifies
\begin{equation}\label{eq:5.30}
\int_{\mathbb{R}^2}v_s(-\Delta)^s\psi_R\,dx+\int_{\mathbb{R}^2}|x|^2v_s\psi_R\,dx
-\int_{\mathbb{R}^2}|v_s|^2v_s\psi_R\,dx-\mu_s\int_{\mathbb{R}^2}v_s\psi_R\,dx=0.
\end{equation}
For $s_k>\frac12$, by the H\"{o}lder inequality, $\|{v}_{s_k}\|_2^2=N$ and the Plancherel theorem,
\begin{equation}\label{eq:5.31}
\begin{split}
\Big|\int_{\mathbb{R}^2}{v}_{s_k}(-\Delta)^{s_k}\psi_R\,dx\Big|
&=R^{-2s_k}\Big|\int_{\mathbb{R}^2}{v}_{s_k}(-\Delta)^{s_k}
\psi(\frac xR )\,dx\Big|\\
&\leq R^{-2s_k}\Big(\int_{\mathbb{R}^2}{v}_{s_k}^2\,dx\Big)^\frac12
\Big(\int_{\mathbb{R}^2}|(-\Delta)^{{s_k}}\psi(\frac xR )|^2\,dx\Big)^\frac12\\
&=R^{1-2s_k}N^{\frac12}\Big(\int_{\mathbb{R}^2}|(-\Delta)^{{s_k}}\psi|^2\,dx\Big)^\frac12\\
&=R^{1-2s_k}N^{\frac12}\Big(\int_{\mathbb{R}^2}{|\xi|}^{2s_k}|\hat{\psi}|^2\,dx\Big)^\frac12\\
&\leq R^{1-2s_k}N^{\frac12}\Big(\int_{\mathbb{R}^2}{|\xi|}^{2}|\hat{\psi}|^2\,dx\Big)^\frac{s_k}2
\Big(\int_{\mathbb{R}^2}|\hat{\psi}|^2\,dx\Big)^\frac{1-s_k}2\\
&\leq CR^{1-2s_k}N^{\frac12}\to0, \quad as  \quad k\to \infty.
\end{split}
\end{equation}
The fact that $\beta\neq0$ and \eqref{eq:5.28} yields that for $k$ large enough, $-\mu_k\geq\frac12\beta^2\varepsilon_k^{-2s_k}$. This with \eqref{eq:5.30} yields
\begin{equation}\label{eq:5.32}
\begin{split}
&\int_{\mathbb{R}^2}v_{k}(-\Delta)^{s_k}(\psi(\frac xR ))\,dx+\frac12\beta^2\varepsilon_k^{-2s_k}\int_{\mathbb{R}^2}v_{k}\psi(\frac xR )\,dx\\
&\leq
\int_{\mathbb{R}^2}v_{k}(-\Delta)^{s_k}(\psi(\frac xR ))\,dx-\mu_{s_k}\int_{\mathbb{R}^2}v_{k}\psi(\frac xR) \,dx\\
&\leq\int_{\mathbb{R}^2}v_{k}^3\psi(\frac xR) \,dx.
\end{split}
\end{equation}
Letting $R\to+\infty$ in \eqref{eq:5.32}, by \eqref{eq:5.31}, we obtain
\begin{equation}\label{eq:5.33}
\begin{split}
\int_{\mathbb{R}^2}v_{k}\,dx&\leq2\beta^{-2}\varepsilon_k^{2s_k}\int_{\mathbb{R}^2}v_{k}^3 \,dx\\
&\leq2\beta^{-2}\varepsilon_k^{2s_k}\Big(\int_{\mathbb{R}^2}v_{k}^4 \,dx\Big)^{\frac12}\Big(\int_{\mathbb{R}^2}v_{k}^2 \,dx\Big)^{\frac12}\\
&=2\beta^{-2}N^{\frac12}\varepsilon_k^{2s_k}\Big(\int_{\mathbb{R}^2}v_{k}^4 \,dx\Big)^{\frac12}.
\end{split}
\end{equation}
Thus, Lemma \ref{lem:5.2} implies
\begin{equation}\label{eq:5.34}
\begin{split}
\int_{\mathbb{R}^2}v_{k}\,dx
&\leq2\beta^{-2}N^{\frac12}\varepsilon_k^{2s_k}\Big(2s\int_{\mathbb{R}^2}|(-\Delta)^{\frac{s_k}{2}}v_{k}|^2 \,dx\Big)^{\frac12}\\
&=2\beta^{-2}N^{\frac12}\Big(\int_{\mathbb{R}^2}|(-\Delta)^{\frac{s_k}{2}}v_{k}|^2 \,dx\Big)^{-1}\Big(2s\int_{\mathbb{R}^2}|(-\Delta)^{\frac{s_k}{2}}v_{k}|^2 \,dx\Big)^{\frac12}\\
&=2\sqrt{2s}\beta^{-2}N^{\frac12}\Big(\int_{\mathbb{R}^2}|(-\Delta)^{\frac{s_k}{2}}v_{k}|^2 \,dx\Big)^{-\frac12}.
\end{split}
\end{equation}
Since $v_k$ is radially decreasing with respect to $|x|$, by \eqref{eq:5.34}, we have
\begin{equation*}
\begin{split}
v_k(x)&=v_k(|x|)\\
&\leq C|x|^{-2}\int_{|y|\leq |x|}v_k(y)\,dy\\
&\leq C|x|^{-2}\int_{\mathbb{R}^2}v_k(y)\,dy\\
&\leq 2C\sqrt{2s}\beta^{-2}N^{\frac12}\Big(\int_{\mathbb{R}^2}|(-\Delta)^{\frac{s_k}{2}}v_{k}|^2 \,dx\Big)^{-\frac12}|x|^{-2}.
\end{split}
\end{equation*}
This and \eqref{eq:5.34} yield that
\begin{equation}\label{eq:5.35}
\begin{split}
\int_{\mathbb{R}^2}|x|^2|v_k|^2\,dx&\leq 2\sqrt{2s}\beta^{-2}N^{\frac12}\Big(\int_{\mathbb{R}^2}|(-\Delta)^{\frac{s_k}{2}}v_{k}|^2 \,dx\Big)^{-\frac12}
\int_{\mathbb{R}^2}v_{k}\,dx\\
&\leq 8sN\beta^{-4}\Big(\int_{\mathbb{R}^2}|(-\Delta)^{\frac{s_k}{2}}v_{k}|^2 \,dx\Big)^{-1}\to 0 \quad as  \quad k\to \infty.
\end{split}
\end{equation}
By \eqref{eq:5.14}, \eqref{eq:5.15} and \eqref{eq:5.35}, we have
\[
\lim_{k\to\infty}{\|(-\Delta)^{\frac {s_k}{2}}v_{s_k}\|_{L^2(\mathbb{R}^2)}^2}{\Big(\frac{N_{s_k}^*}{N}\Big)^{-\frac{s}{1-s}}}=N
\]
and
\begin{equation}\label{eq:5.36}
\lim_{k\to\infty}\frac{\int_{\mathbb{R}^2}|v_k|^4\,dx}{\int_{\mathbb{R}^2}|(-\Delta)^{\frac{s_k}{2}}v_k|^2\,dx}=2.
\end{equation}
This and \eqref{eq:5.27} yield that $\mu_k\varepsilon_k^{2s_k}\to -N^{-1}$, and then $\beta^2=N^{-1}$.
The proof is complete.\quad $\Box$

\bigskip

\bigskip

\section{Further discussion}

\bigskip

In this section,   we discuss the  attractive fractional Gross-Pitaevskii equation
 \begin{equation}\label{eq:6.1}
i\psi_t(x,t)=(-\Delta)^s\psi(x,t)+V(x)\psi(x,t)-|\psi|^2\psi(x,t),\ \  (x,t)\in \mathbb{R}^2\times\mathbb{R}
\end{equation}
with  the general potential $V(x)$ in two dimensions. A standing wave solution $\psi(x,t)=e^{-i\mu_s t}u(x)$ of equation \eqref{eq:6.1}
fulfills that
\begin{equation}\label{eq:6.2}
(-\Delta)^su+V(x)u= |u|^2u+\mu_s u.
\end{equation}
Its corresponding energy functional becomes
\begin{equation}\label{eq:6.3}
E_{N,s}^V(u)=\frac{1}{2}\int_{\mathbb{R}^2}\big(|(-\Delta)^{\frac s2} u(x)|^2+V(x) |u(x)|^2\big)\,dx-\frac{1 }{4}\int_{\mathbb{R}^2}|u(x)|^4\,dx.
\end{equation}

\bigskip

We assume that

$(V_0)$ $V(x)\geq 0, \quad \lim_{|x|\to +\infty}V(x)=+\infty$.

\bigskip

The hypothesis $(V_0)$ allows us to show as the proof of Lemma \ref{lem:2.7}  that

\begin{Lemma}\label{lem:6.1}Let $\mathcal{H}_{s,V}$ be defined in \eqref{eq:2.1a}. The inclusion $\mathcal{H}_{s,V}\hookrightarrow L^p(\mathbb{R}^2)$ is compact for $ p\in[2,\frac 2{1-s})$.
\end{Lemma}

\bigskip

Define
\[
S^V(N)=\{u\in\mathcal{H}_{s,V}:\int_{\mathbb{R}^2}|u|^2\,dx=N\}.
\]
and
\[
A^V_{t_s}=\{u\in S^V(N):\int_{\mathbb{R}^2}|(-\Delta)^{\frac s2}u|^2\,dx\leq t_s\},
\]
where $t_s$ is defined in \eqref{eq:1.19}.

With the  help of Lemma \ref{lem:6.1},
we may prove  in the same way as the proof of Theorem \ref{thm:1} that

\begin{Theorem}\label{thm:6.1}
Suppose $(V_0)$ and $0<N<N^*$. There exists $\varepsilon_N>0$  such that for any $s\in (1-\varepsilon_N,1)$, $E^V_{N,s}(u)$  admits a nonnegative   minimizer $u_s$ in $A_{t_s}^V$, that is
\begin{equation}\label{eq:6.4}
E^V_{N,s}(u_s)=\inf_{u\in A_{t_s}^V}E^V_{N,s}(u)\quad ,\quad  E^V_{N,s}(u_s)<\inf_{u\in \partial A_{t_s}^V}E^V_{N,s}(u),
\end{equation}
and $u_s$ is a weak solution of  \eqref{eq:6.2}.
\end{Theorem}

\bigskip

If we  assume in addition that

$(V_1)$ $V$ is radially symmetric and increasing with respect to $|x|$;

$(V_2)$ $V\in C^1(\mathbb{R}^n,\mathbb{R})$ and $|x\cdot \nabla V|\leq C_1(1+|V|)$ with $C_1>0$;


\bigskip

 Replacing \eqref{eq:3.13} and \eqref{eq:3.14a} in the proof of Theorem \ref{thm:2} by
\begin{equation}\label{eq:6.4a}
s\int_{\mathbb{R}^2}|(-\Delta)^{\frac{s}{2}}v_s|^2\,dx-\int_{\mathbb{R}^2}(x\cdot \nabla V)|v_s|^2\,dx
-\frac12\int_{\mathbb{R}^2}|v_s|^4\,dx=0.
\end{equation}
and
\begin{equation}\label{eq:6.4b}
\begin{split}
2E_{N,s}^V(v_s)&=\int_{\mathbb{R}^2}\big(|(-\Delta)^{\frac s2}v_s|^2+|x|^2|v_s|^2\big)\,dx-\frac12\int_{\mathbb{R}^2}|v_s|^4\,dx\\
&=(1-s)\int_{\mathbb{R}^2}|(-\Delta)^{\frac s2}v_s|^2\,dx+\int_{\mathbb{R}^2}(V+\frac12x\cdot \nabla V)|v_s|^2\,dx\\
&=(1-s)\int_{\mathbb{R}^2}|(-\Delta)^{\frac s2}v_s|^2\,dx+\int_{\mathbb{R}^2}(V+\frac12|x|V_r)|v_s|^2\,dx\\
&\leq 2E_{N,s}^V(u_s)\leq C,
\end{split}
\end{equation}
respectively,  we may deduce as  Theorem \ref{thm:2} that

\begin{Theorem}\label{thm:6.2} Assume $V$ satisfies   $(V_0)$-$(V_2)$. Let $u_s$   be a nonnegative local minimizer of $E^V_{N,s}(u)$ inside $A_{t_s}$.
Suppose $0<N<N^*$. There exists $\varepsilon_N>0$  such that for  $s\in (1-\varepsilon_N,1)$, we have that

$(i)$ $u_s$ is radially symmetric and decreasing with respect to $|x|$;

$(ii)$ $u_s$ is  a ground state of $E^V_{N,s}(u)$ on $S^V(N)$;

$(iii)$ The following uniform bounds
\begin{equation}\label{eq:6.6}
\int_{\mathbb{R}^2}|(-\Delta)^{\frac s2}u_s|^2\,dx\leq C
\end{equation}
and
\begin{equation}\label{eq:6.7}
\int_{\mathbb{R}^2}V(x)|u_s|^2\,dx\leq \frac C2
\end{equation}
hold, where $C>0$ is independent of $s$.
\end{Theorem}

Assuming further that

$(V_3)$   $x\cdot \nabla V\geq C_2V-C_3$ with $C_2>0$ and $C_3\geq 0$.

Using the assumption $(V_3)$, we may prove the nonexistence of local minimizers of $E_{N,s}^V(u)$ in $A_{t_s}$ by a  method different from that of  Theorem \ref{thm:3}, such a  method does not depend the increase order of  $V$ as $x\to\infty$.

\begin{Theorem}\label{thm:6.3}Assume $V$ satisfies   $(V_0)$-$(V_3)$.
If $N>N^*$, there exists $\varepsilon^1_{N}>0$ such that for any $s\in (1-\varepsilon^1_{N},1)$, $E_{N,s}^V(u)$  admits no local positive minimizer in $A_{t_s}^V$.
\end{Theorem}
\begin{proof} Let $N>N^*$. Recall that $t_s\to0$ as $s\to1_-$.
Similar to \eqref{eq:3.24}, we derive
\begin{equation}\label{eq:6.8}
\begin{split}
\int_{\mathbb{R}^2}|(-\Delta)^{\frac38}u_s|^2\,dx&\leq \Big(\int_{\mathbb{R}^2}|(-\Delta)^{\frac s{2}}u_s|^2\,dx\Big)^{\frac3{4s}} \Big(\int_{\mathbb{R}^2}|u_s|^2\,dx\Big)^{1-\frac3{4s}}\\
&\leq t_s^{\frac3{4s}}N^{1-\frac3{4s}}\to 0 \
\end{split}
\end{equation}
as $s\to1_-.$

We remark that $u_s$ satisfies \eqref{eq:6.4a}. Combining this fact with  $x\cdot \nabla V\geq C_2V-C_3$, we have
\begin{equation}\label{eq:6.9}
\begin{split}
C_2\int_{\mathbb{R}^2}V|u_s|^2\,dx&\leq \int_{\mathbb{R}^2}(x\cdot \nabla V)|u_s|^2\,dx+ C_3\int_{\mathbb{R}^2}|u_s|^2\,dx\\
&\leq s\int_{\mathbb{R}^2}|(-\Delta)^{\frac s{2}}u_s|^2\,dx+C_3N\\
&\leq st_s+C_3N\to C_3N
\end{split}
\end{equation}
as $s\to1_-.$

In accordance with  \eqref{eq:6.8}, \eqref{eq:6.9} and Lemma \ref{lem:6.1}, there exists $\{s_k\}$ with $s_k\to 1_-$ as $k\to\infty$ such that
\[
u_{s_k}\rightharpoonup u \quad in \quad \mathcal{H}_{s,V};
\]
\[
u_{s_k}\to u \quad in \quad L^2(\mathbb{R}^2).
\]
Moreover, $\|u\|_2^2=N$ and
\[
\int_{\mathbb{R}^2}|(-\Delta)^{\frac38}u|^2\,dx\leq\liminf_{k\to \infty}\int_{\mathbb{R}^2}|(-\Delta)^{\frac38}u_{s_k}|^2\,dx=0.
\]
This with $u\in L^2(\mathbb{R}^2)$ yields that $u=0$, which contradicts $\|u\|_2^2=N$.
\end{proof}

\bigskip

Assuming $V$ satisfies $(V_0)-(V_3)$ and  $\limsup_{x\to \infty}|x|^{-2}V(x)\leq C$ and $C>0$, we can deduce in  the same routine of the proof of Theorem \ref{thm:4}--\ref{thm:6} that corresponding results hold for
$E^V_{N,s}(u)$ and \eqref{eq:6.2}.

\vspace{2mm} \noindent{\bf Conflict of Interest} The authors declare that they have no conflict of interest.

\vspace{2mm} \noindent{\bf Funding}
Jinge Yang was supported by NNSF of China, No:12361025 and  by Jiangxi
Provincial Natural
Science Foundation, No:20232BAB201002 and No:20212ACB201003.
Jianfu Yang is supported by NNSF of China, No:12171212.

\vspace{2mm} \noindent{\bf Data Availability}  Data sharing not applicable to this article as no datasets were generated or analyzed during the current study.

{\small

\end{document}